\documentclass[12pt,a4paper,dk,reqno]{amsart}
\usepackage{amsmath,amsfonts,amssymb,amsthm,amscd}
\usepackage[english]{babel}
\usepackage[mathscr]{eucal}
\usepackage{graphicx}

\numberwithin{equation}{section}
\renewcommand{\a}{\alpha}
\renewcommand{\b}{\beta}
\newcommand{\g}{\gamma}
\newcommand{\G}{\Gamma}
\renewcommand{\d}{\delta}
\newcommand{\D}{\Delta}
\newcommand{\h}{\chi}

\renewcommand{\o}{\omega}
\renewcommand{\O}{\Omega}
\renewcommand{\r}{\rho}

\renewcommand{\t}{\tau}
\renewcommand{\th}{\theta}
\newcommand{\e}{\varepsilon}
\newcommand{\f}{\varphi}
\newcommand{\F}{\Phi}
\newcommand{\x}{\xi}

\newcommand{\y}{\eta}

\newcommand{\p}{\psi}

\newcommand{\C}{{\Bbb C}}

\newcommand{\N}{{\Bbb N}}
\newcommand{\R}{{\Bbb R}}

\newcommand{\T}{{\mathbf T}}

\newcommand{\curl}{{\rm curl}\,}

\newcommand{\diver}{{\rm div}\,}
\newcommand{\inte}{{\rm int}\,}
\newcommand{\pd}{\partial}

\newcommand{\supp}{\operatorname{supp\,}}

\newcommand{\loc}{\operatorname{{loc}}}

\newtheorem{theorem}{Theorem}[section]
\newtheorem{proposition}[theorem]{Proposition}
\newtheorem{lemma}[theorem]{Lemma}

\theoremstyle{definition}
\newtheorem{definition}[theorem]{Definition}
\newtheorem{assumption}[theorem]{Assumption}

\theoremstyle{remark}
\newtheorem{remark}[theorem]{Remark}
\newtheorem{example}[theorem]{Example}

\newcommand{\dss}{{\mathbf{ds}}}
\newcommand{\na}{\nabla}

\begin{document}


\title[Two Dimensional Incompressible Ideal Flow]
      {Two Dimensional Incompressible Ideal Flow Around a Thin Obstacle Tending to a Curve}
\author[C. Lacave]{Christophe Lacave}
\address[C. Lacave]{Universit\'e de Lyon\\
Universit\'e Lyon1\\
CNRS, UMR 5208 Institut Camille Jordan\\
Batiment du Doyen Jean Braconnier\\
43, blvd du 11 novembre 1918\\
F - 69622 Villeurbanne Cedex\\
France}

\email{lacave@math.univ-lyon1.fr}

\date{\today}

\begin{abstract}

In this work we study the asymptotic behavior of solutions of the incompressible two-dimensional Euler equations in the exterior of a single smooth obstacle when the obstacle becomes very thin tending to a curve. We extend results by Iftimie, Lopes Filho and Nussenzveig Lopes, obtained in the context of an obstacle tending to a point, see [Comm. PDE, {\bf 28} (2003), 349-379].

\end{abstract}

\maketitle

\section{Introduction}

The purpose of this work is to study the influence of a thin  material
obstacle on the behavior of two-dimensional incompressible ideal
flows. More precisely, we consider a family of
obstacles $\O_\e$ which are smooth, bounded, open, connected, simply connected
subsets of the plane, contracting to a smooth curve $\G$ as $\e\to
0$. 
Given the geometry of the exterior domain $\R^2\setminus\O_\e$, a
velocity field (divergence free and tangent to the boundary) on this
domain is uniquely determined by the two following (independent)
quantities: vorticity and circulation of velocity on the boundary of
the obstacle. Throughout this paper we assume that initial vorticity 
$\o_0$ is independent of $\e$, smooth,
compactly supported outside the obstacles $\O_\e$ and that $\g$, the
circulation of the initial velocity on the boundary, is independent of
$\e$. From  the work of K. Kikuchi \cite{kiku}, we know that there
exists  $u^\e=u^\e(x,t)$ a unique global solution to the Euler
equation in the exterior domain $\R^2\setminus\O_\e$ associated to the
initial data described above. Our aim is to determine the limit of
$u^\e$ as $\e\to0$. As a consequence, we also obtain the existence of a solution of the
 Euler equations in the exterior of the curve $\G$.

The study of incompressible fluid flows in presence of small obstacles
was initiated by Iftimie, Lopes Filho and Nussenzveig Lopes
\cite{ift_lop,ift_lop_2}. The paper \cite{ift_lop} treats the same problem as
above but with obstacles that shrink homothetically to a point $P$, instead
of a curve. The case of Navier-Stokes is considered in
\cite{ift_lop_2}. In the inviscid case, these authors prove that if
the circulation $\g$ vanishes, then the
limit velocity verifies the Euler equation in $\R^2$ (with the same
initial vorticity). If the circulation is non-zero, then the limit
equation involves a new term that looks that a fixed Dirac mass in the
point $P$ of strength $\g$; the initial vorticity also acquires a
Dirac mass in $P$. In the case of Navier-Stokes, the limit equation is
always Navier-Stokes but the initial vorticity of the limit equation
still has an additional Dirac mass in $P$.

Here we will show that, in the inviscid case, the limit equation is
the Euler equation in $\R^2\setminus\Gamma$. The initial velocity for
the limit equation is a velocity field which is divergence free in
$\R^2$, tangent to $\G$ such that the curl computed in
$\R^2\setminus\G$ is $\o_0$ and the curl computed in
$\R^2$ is   $\o_0+g_{\o_0}\delta_\G$ where $g_{\o_0}$ is a density
given explicitly in terms of $\o_0$ and $\g$. Alternatively,
$g_{\o_0}$ is the jump of the tangential velocity across $\G$.

  More precisely, let $\F_\e$ be a
cut-off function in a small $\e$-neighborhood of the boundary (the
precise definition of $\F_\e$ is given in Subsection \ref{cutoff}) and
set $\o_\e=\curl u^\e$. Our main Theorem may be stated as follows:
\newpage 
\begin{theorem} \label{1.1}
There exists a subsequence $\e=\e_k\to 0$ such that
\begin{itemize}
\item[(a)] $\F^\e u^\e \to u$ strongly in $L^2_{\loc}(\R_+\times\R^2)$;
\item[(b)] $\F^\e \o^\e \to \o$ weak$*$ in $L^\infty(\R_+;L^4_{\loc}(\R^2))$;
\item[(c)] $u$ is related to $\o$ by means of relation \eqref{uomega} 
\item[(d)] $u$ and $\o$ are weak solutions of $\o_t+u.\na\o=0$ in $\R^2\times(0,\infty)$.
\end{itemize}
\end{theorem}

The limit velocity $u$ is explicitly given in terms of $\o$ and $\g$ (see Theorem \ref{5.6}) and can be viewes as the divergence free vector field which is tangent to $\G$, vanishing at infinity, with curl in $\R^2\setminus\G$ equal to $\o$ and with circulation around the curve $\G$ equal to $\g$. This velocity is  blowing up at
 the endpoints of the curve $\G$ as the inverse of the square root of
 the distance. and has a jump across $\G$. Moreover, we have $\curl u=\o+g_\o(s)\d_\G$ in $\R^2\times[0,\infty)$, where $\d_\G$ is the Dirac function of the curve $\G$, and the 
$g_\o$ which is defined in Lemma \ref{5.9} depends on $\o$ and the circulation $\g$. 
 The function $g_\o$ is continuous on $\G$ and blows up at
 the endpoints of the curve $\G$ as the inverse of the square root of
 the distance. One can also characterize $g_\o$ as the jump of the
 tangential velocity across $\G$. The presence of the additional term $g_\o$ in the expression of $\curl u$, compared of the Euler equation in the full plane, is compulsory to obtain a vector field tangent to the curve, with a circulation $\g$ around the curve.

There is a sharp contrast between the behavior of ideal flows around a
small and thin obstacle. In \cite{ift_lop}, the authors studied the
vanishing obstacle problem when the obstacle tends homothetically to a
point $P$. Their main result is that the limit vorticity satisfies a
modified vorticity equation of the form $\o_t+u.\na\o=0$, with $\diver
u=0$ and $\curl u=\o+\g\d(x-P)$. In other words, for small obstacles
the correction due to the vanishing obstacle appears as
time-independent additional convection centered at $P$, whereas in the
thin obstacle case, the correction term depends on the time. Although
treating a related problem, the present work requires a different
approach. Indeed, in \cite{ift_lop}, the proofs are simplified by the
fact that the obstacles are  homothetic  to a fixed domain. Indeed, an
easy change of variables $y=x/\e$ allows in that case to return to a
fixed obstacle and to deduce the required estimates.
This argument clearly does not work here and a considerable amount of
work is needed to characterize the conformal mapping that sends the
exterior of a small obstacle into the exterior of the unit disk. 
Moreover, in \cite{ift_lop} the authors use the div-curl Lemma to
obtain strong convergence for velocity. This is made possible by the
validity of some bounds on the divergence and the curl of the
velocity. A consequence of our work is that these estimates are no
longer valid in our case, so this approach can not work. We will be
able to prove directly strong convergence for the velocity through several
applications of the Lebesgue dominated convergence theorem.  We
finally observe that, in contrast to the case of \cite{ift_lop}, the
vanishing of the circulation $\g$ plays no role in our result. The
limit velocity will always verify the same type of equation.

We also mention that  Lopes Filho treated in \cite{lop} the case of
several obstacles with one of the obstacles tending to a point, but the author had to work on a bounded domain. In this case, we do not have explicit formulas anymore, and the conformal mapping technique is replaced by qualitative analysis using elliptic techniques, including variational methods and the maximum principle.

The remainder of this work is organized in five sections. We introduce in Section 2 a family of conformal mappings between the exterior of $\O_\e$ and the exterior of the unit disk, allowing the use of explicit formulas for basic harmonic fields and the Biot-Savart law, which will be really helpful to obtain sharp estimations. In the third part, we precisely formulate the flow problem in the exterior of a vanishing obstacle. In Section 4, we collect {\it a priori} estimates in order to find the equation limit in the Section 5 of this article. The last subsection concerns an existence result of the Euler equations on the exterior of a curve.

For the sake of clarity, the main notations are listed in an appendix at the end of the paper.

\newpage

\section{The laplacian in an exterior domain}

\subsection{Conformal maps} \

Let $D=B(0,1)$ and $S=\pd D$.
In what follows we identify $\R^2$ with the complex plane $\C$.

We begin this section by recalling some basic definitions on the curve.

\begin{definition} We call a {\it Jordan arc} a curve $C$ given by a parametric representation $C:\f(t)$, $0\leq t\leq 1$ with $\f$ an injective ($=$ one-to-one) function, continuous on $[0,1]$. An {\it open Jordan arc} has a parametrization $C:\f(t)$, $0< t<1$ with $\f$ continuous and injective on $(0,1)$.

We call a {\it Jordan curve} a curve $C$ given by a parametric representation $C:\p(t)$, $t\in \R$, $1$-periodic, with $\p$ an injective function on $[0,1)$, continuous on $\R$.
\end{definition}

Thus a Jordan curve is closed ($\f(0)=\f(1)$) whereas a Jordan arc has distinct endpoints. If $J$ is a Jordan curve in $\C$, then the Jordan Curve Theorem states that $\C\setminus J$ has exactly two components $G_0$ and $G_1$, and these satisfy $\pd G_0=\pd G_1=J$.

The Jordan arc (or curve) is of class $C^{n,\a}$ ($n\in\N^*,0<\a\leq 1$) if its parametrization $\f$ is $n$ times continuously differentiable, satisfying $\f'(t)\neq 0$ for all $t$, and if $|\f^{(n)}(t_1)-\f^{(n)}(t_2)|\leq C|t_1-t_2|^\a$ for all $t_1$ and $t_2$.

Let $\G: \G(t),0\leq t\leq 1$ be a Jordan arc. Then the subset $\R^2\setminus\G$ is connected and we will denote it by $\Pi$. The purpose of this part is to obtain some properties of a biholomorphism $T: \Pi \to \inte\ D^c$.

\begin{proposition}\label{2.2}
If $\G$ is a $C^2$ Jordan arc, such that the intersection with the segment $[-1,1]$ is a finite union of segments and points, then there exists a biholomorphism $T:\Pi\to \inte\ D^c$ which verifies the following properties:
\begin{itemize}
\item $T^{-1}$ and $DT^{-1}$ extend continuously up to the boundary, and $T^{-1}$ maps $S$ to $\G$,
\item $DT^{-1}$ is bounded,
\item $T$ and $DT$ extend continuously up to $\G$ with different values on each side of $\G$, except at the endpoints of the curve where $T$ behaves like the square root of the distance and $DT$ behaves like the inverse of the square root of the distance,
\item $DT$ is bounded in the exterior of the disk $B(0,R)$, with $\G \subset B(0,R)$,
\item $DT$ is bounded in $L^p(\Pi\cap B(0,R))$ for all $p<4$ and $R>0$.
\end{itemize}
\end{proposition}

\begin{proof}

We first study the case where the arc is the segment $[-1,1]$. We can have here an explicit formula for $T$. Indeed, the Joukowski function 
$$G(z)= \frac{1}{2} (z+\frac{1}{z})$$
is a biholomorphism between the exterior of the disk and the exterior of the segment. It maps the circle $C(0,R)$ on the ellipse parametrized by $\frac{1}{2}(R+1/R)\cos\th+\frac{1}{2}(R-1/R)i \sin\th$ with $\th\in [0,2\pi)$, and it maps the unit circle on the segment.  

Remarking that $G(z)=G(1/z)$ we can conclude that $G$ is also a biholomorphism between the interior of the disk minus $0$ and the exterior of the segment.

Therefore, any $z\notin [-1,1]$ has one antecedent of $G$ in $D$ and another one in $\inte\ D^c$. For $z\in [-1,1]$ the antecedents are $\exp(\pm i \arccos z)=z\pm i\sqrt{1-z^2}$. Therefore, there are exactly two antecedents except for $-1$ and $1$. In fact, we have considered the double covering $G$ from $\C^*$ to $\C$, which is precisely ramified over $-1$ and $1$. 

Let $\tilde T$ be the biholomorphism between the exterior of the segment and $\inte\ D^c$, such that $\tilde T^{-1}=G$, and let $\tilde T_{\inte}=1/\tilde T$ be the biholomorphism between the exterior of the segment and $D\setminus\{0\}$, such that $\tilde T^{-1}_{\inte}=G$.

To find an explicit formula of $\tilde T$, we have to solve an equation of degree two. We find two solutions:
\begin{equation}
\label{Tpm}
\tilde T_+ = z+\sqrt{z^2-1} \text{\ \ and\ \ } \tilde T_- = z-\sqrt{z^2-1}.
\end{equation}

We consider that the function square-root is defined by $\sqrt{z} = \sqrt{|z|}e^{i\th/2}$ with $\th$, the argument of $z$, verifying $-\pi<\th\leq\pi$. It is easy to observe that $\tilde T = \tilde T_+$ on $\{z \mid \Re(z)>0\}\cup i\R_+$ and $\tilde T = \tilde T_-$ on $\{z \mid  \Re(z)<0\}\cup i\R_-$. Despite this, $\tilde T$ is $C^\infty(\C\setminus[-1,1],\inte D^c)$ because $\tilde T=G^{-1}$.

Therefore in the segment case, $T=\tilde T$ and the first two points are straightforward. An obvious calculation allows us to find an explicit formula for $\tilde  T'$:
\begin{equation}
\label{T'}
\tilde T'(z)=1\pm \frac{z}{\sqrt{z^2-1}},
\end{equation}
with the choice of sign as above. This form shows us that $D\tilde T$ blows up at the endpoints of the segment as the inverse of the square root of the distance, which is bounded in $L^p_{\loc}$ for $p<4$. Moreover, a mere verification can be done to find that for every $x\in (-1,1)$, we have
$$\lim_{z\to x, \Im (z)>0} T(z)=x+i\sqrt{1-x^2}=\tilde T_+(x)$$
even if $\Re (z)<0$, and 
$$\lim_{z\to x, \Im (z)<0} T(z)=x-i\sqrt{1-x^2}=\tilde T_-(x).$$
In the same way, $DT$ extends continuously up to each side of $\G$, which concludes the Lemma in the segment case.

\vspace{12pt}

Now, we come back to our problem, with any curve $\G$. After applying a homothetic transformation, a rotation and a translation, we can suppose that the endpoints of the curve are $-1=\G(0)$ and $1=\G(1)$. Next, we consider the curve $\tilde \G \equiv \overline{\tilde T(\G)}\cup \overline{\tilde T_{\inte}(\G)}=\tilde T_+(\G)\cup \tilde T_-(\G)$. We now show that $\tilde\G$ is a $C^{1,1}$ Jordan curve.

We consider first the case where $\G$ does not intersect the segment $(-1,1)$. Then $\g_1\equiv\tilde T(\G)\subset D^c$ and $\g_2\equiv\tilde T_{\inte}(\G)\subset D$ are $C^1$ open Jordan arcs, with the endpoints on $-1$ and $1$ (see Picture 1). So $\tilde T(-1)=-1=\tilde T_{\inte}(-1)$ and we can observe that $\tilde \G$ is a Jordan curve.

\begin{figure}[ht]
                \begin{minipage}[t]{4cm}
\includegraphics[height=5.2cm]{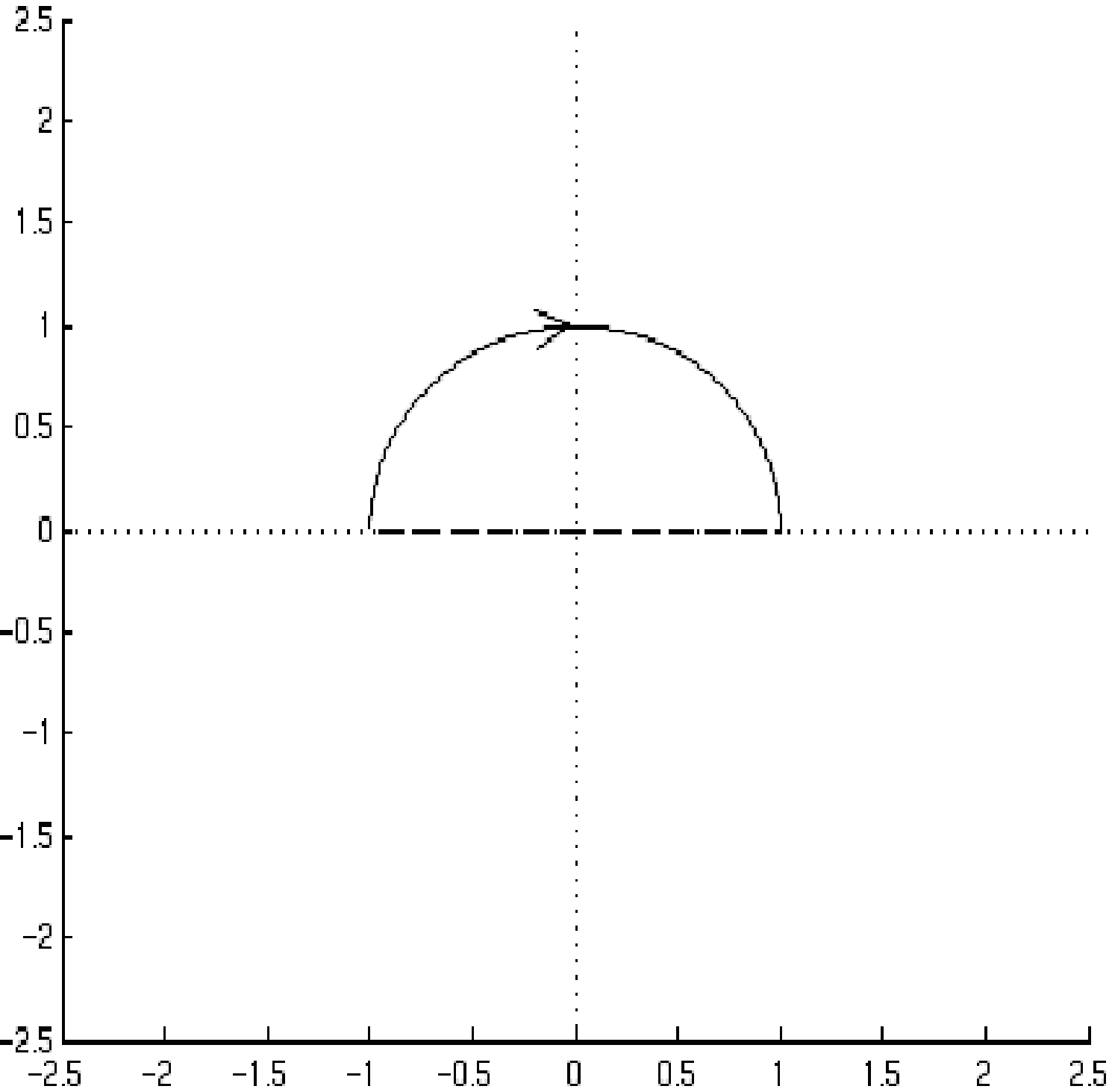}

\centerline{\ \hspace{1.2cm}$\G$}

                \end{minipage}
                \hfill
   \hspace{-1cm}             \begin{minipage}[t]{4cm}
\includegraphics[height=5.2cm]{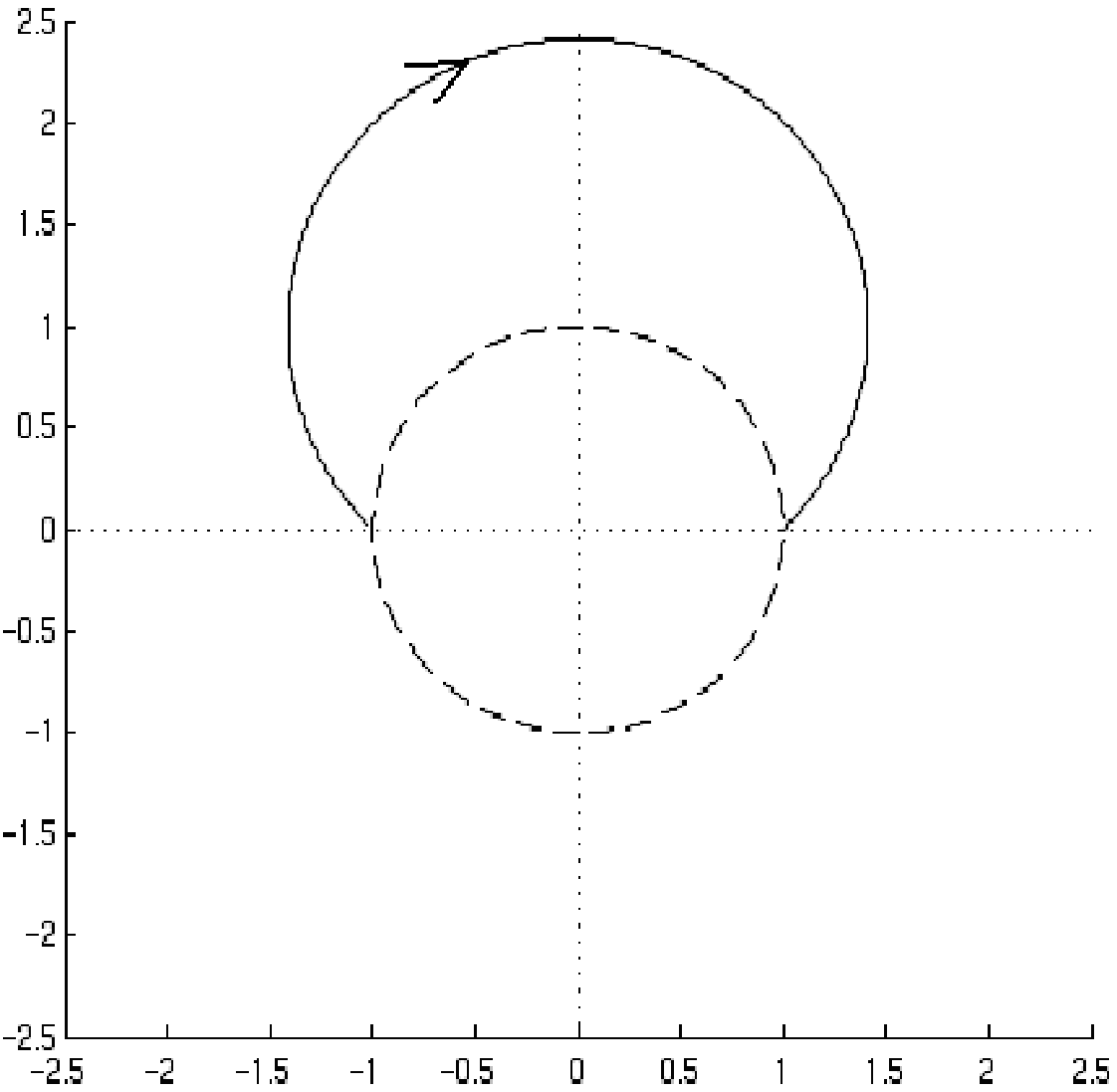}
              
\hspace{1.7cm} {$\g_1=\tilde T(\G)$}

                \end{minipage}
                \hfill
                \begin{minipage}[t]{4cm}
\hspace{-1.5cm}
\includegraphics[height=5.2cm]{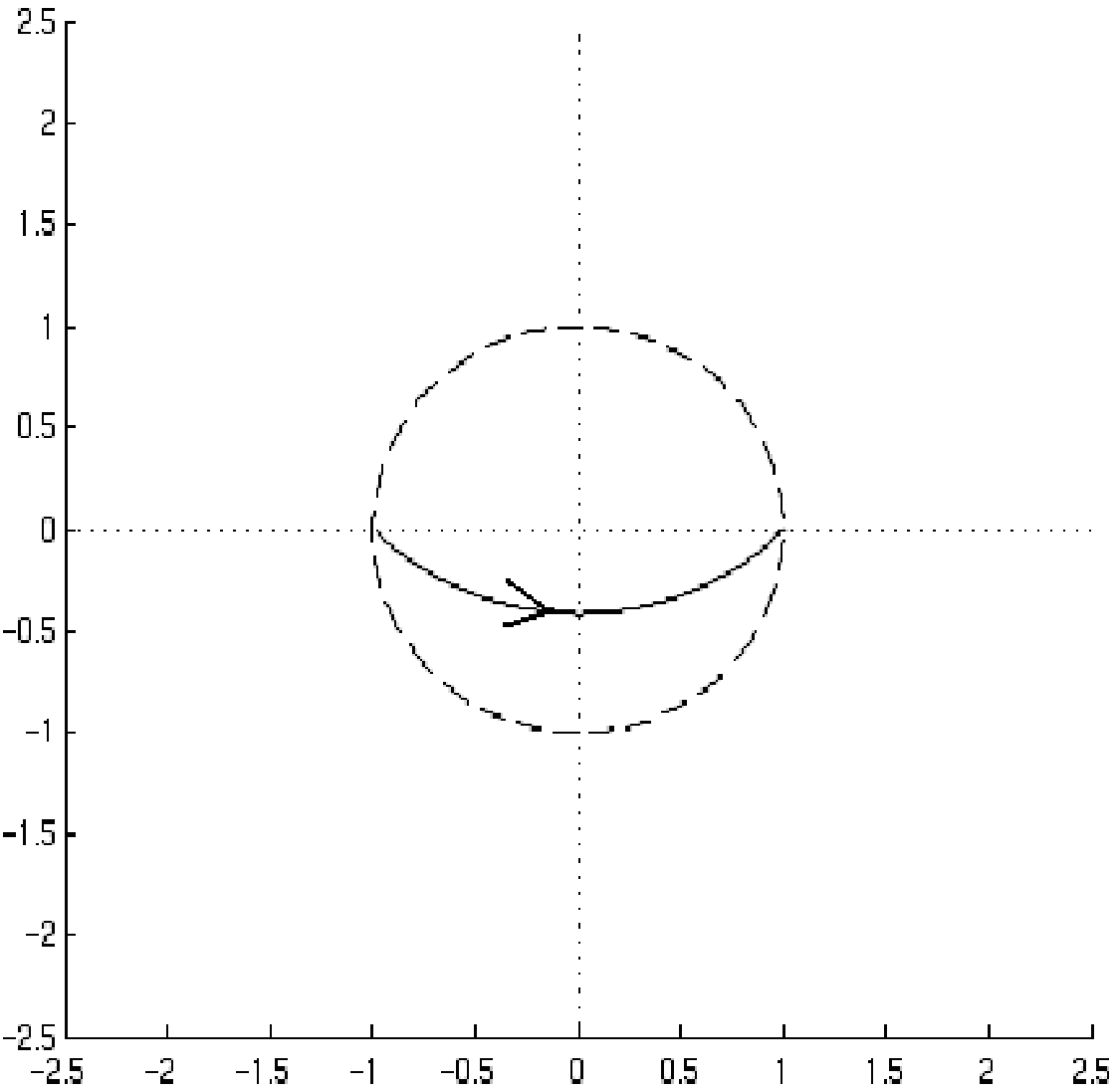}
\centerline{$\g_2=\tilde T_{\inte}(\G)$\ \ \ }

                \end{minipage}
                \end{figure}

\centerline{{\bf PICTURE 1}: $\G$ does not intersect $[-1,1]$}

\newpage

\begin{figure}[ht]
                \begin{minipage}[t]{4cm}
\vspace{-0.1cm}
\includegraphics[height=7cm]{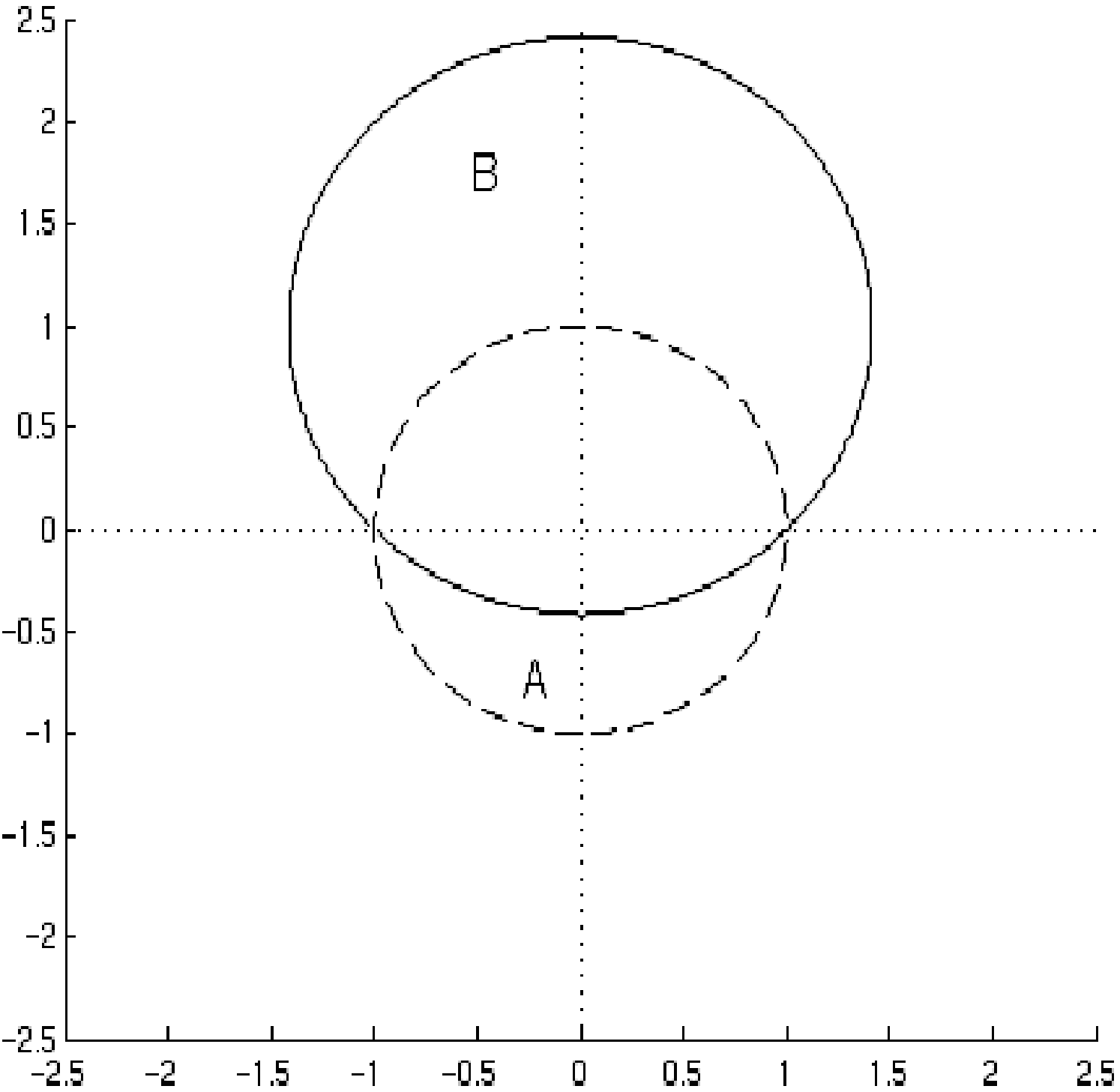}

$\hbox{\ \hspace{3cm} $\tilde\G=G^{-1}(\G)$}$

                \end{minipage}
                \hfill
                \begin{minipage}[t]{8cm}
We wrote {\it open} Jordan curve because the problem with $-1$ and $1$ is that $\tilde T'(\pm 1)$ is not defined. However, if we use the arclength coordinates
\begin{equation}
\label{curviligne}
s(t)=\int_0^t|\g_1'(\t)|d\t=\int_0^t|\tilde T'(\G(\t))||\G'(\t)|d\t,
\end{equation}
which are well-defined and bounded because $\tilde T'$ is bounded in $L^1_{\loc}$, then we have $\dfrac{d\g_1}{ds}=\dfrac{\g'_1 }{|\g_1'|}$. So to prove the derivative continuity, we should show that $\lim_{t\to 0}\dfrac{\g'_1 }{|\g_1'|}$ and  $\lim_{t\to 0}\dfrac{\g'_2 }{|\g_2'|}$ exist and are opposite. For that, we will use the following Lemma :
\end{minipage}\end{figure}

\begin{lemma} \label{2.3} If there exists a neighborhood of $0$ where $\G(t)$ does not intersect the segment $(-1,1)$, then $\dfrac{\tilde T'(\G)}{|\tilde T'(\G)|}(t)$ has a limit as $t\to 0$.
\end{lemma}

\medskip \hspace{-6mm} {\it Proof of the lemma.} First, since $\tilde T'(z)=1-z/\sqrt{z^2-1}$ in a neighborhood of $-1$, we compute
$$\dfrac{\tilde T'(\G)}{|\tilde T'(\G)|}(t)=\dfrac{|\sqrt{\G^2-1}|}{\sqrt{\G^2-1}}(t)\dfrac{\sqrt{\G^2-1}-\G}{|\sqrt{\G^2-1}-\G|}(t).$$
The second fraction tends to 1 as $t\to 0$. We have $\G'(0)\neq 0$, so we can write $\G^2(t)=1+at+o(t)$ for $a\in \C^*$.
If $a\notin \R^-$ then for $t$ small enough $\{\G^2(t)-1\}\subset \C\setminus\R^-$ and 
$$\lim_{t\to 0}\dfrac{\tilde T'(\G)}{|\tilde T'(\G)|}(t)=\frac{\sqrt{|a|}}{\sqrt{a}}=e^{-i\th/2}\text{\ with\ }\th\equiv \arg a\in (-\pi,\pi).$$

For $a\in\R^-$, as $a=(\G^2)'(0)=2\G(0)\G'(0)$, then we have $\G'(0)\in\R^+$ and the curve is tangent to the segment $[-1,1]$. We have here two cases :
\begin{itemize}
\item if $\G$ is over the segment on the neighborhood, then $\Im(\G^2(t)-1)<0$ and  $$\lim_{t\to 0}\dfrac{\tilde T'(\G)}{|\tilde T'(\G)|}(t)= i,$$

\item if $\G$ is below the segment on the neighborhood, then $\Im(\G^2(t)-1)>0$ and  $$\lim_{t\to 0}\dfrac{\tilde T'(\G)}{|\tilde T'(\G)|}(t)= -i.$$
\end{itemize}

{\centerline{\hfill $\square$}\medskip}

\vspace{12pt}

Let us continue the proof of Proposition \ref{2.2}. Lemma \ref{2.3} allows us to observe that $\tilde \G$ is $C^1$ in $-1$, because
$$\lim_{t\to 0}\dfrac{\g'_1}{|\g_1'|}(t)=\lim_{t\to 0}\dfrac{\tilde T'(\G)}{|\tilde T'(\G)|}(t)\dfrac{\G'}{|\G'|}(t)=-\lim_{t\to 0}-\dfrac{|\tilde T^2(\G)|}{\tilde T^2(\G)}(t)\dfrac{\tilde T'(\G)}{|\tilde T'(\G)|}(t)\dfrac{\G'}{|\G'|}(t)=-\lim_{t\to 0}\dfrac{\g'_2}{|\g_2'|}(t),$$ because $\tilde T_{\inte}=1/\tilde T$.

To prove that $\tilde \G'$ is Lipschitz, we will show that $\g'_1$ and $\g'_2$ are $C^1$ with the arclength parametrization denoted by $s$ and defined in (\ref{curviligne}) ($t$ denotes the variable of the original parametrization). Let $f_1(s)=\frac{d\g_1}{ds}(s)= \frac{\g'_1(t)}{|\g_1'(t)|}$, where the primes denote derivatives with respect to $t$, then we need to prove that $\frac{d f_1}{ds}$ has a limit when $s\to 0$. We have
$$\frac{d f_1}{ds}(s) = \frac{\g''_1}{|\g_1'|^2}- \frac{\g'_1}{|\g_1'|^4}<\g'_1,\g''_1> = \frac{1}{|\g_1'|^2}\Bigl(\g''_1- \frac{\g'_1}{|\g_1'|}\langle\frac{\g'_1}{|\g_1'|},\g''_1\rangle\Bigl) \equiv \frac{1}{|\g_1'|^2}A.$$

We compute 
\begin{eqnarray*}
\g'_1&=&\tilde T'(\G)\G',\\
\g''_1&=&\tilde T''(\G)(\G')^2+\tilde T'(\G)\G'',
\end{eqnarray*}
with 
\begin{eqnarray*}
\tilde T(z) &=& z-\sqrt{z^2-1},\\
\tilde T'(z) &=& 1-z/\sqrt{z^2-1},\\
\tilde T''(z) &=& -1/\sqrt{z^2-1} + z^2/\sqrt{z^2-1}^3.
\end{eqnarray*}
We do some Taylor expansions in a neighborhood of zero:
\begin{eqnarray*}
\G(t)&=&-1+at+bt^2+O(t^3), \\
\G^2(t)&=&1-2at+(a^2-2b)t^2+O(t^3), \\
1/\sqrt{\G^2-1}(t)&=&\frac{1}{\sqrt{-2a}}\frac{1}{\sqrt{t}}(1-t(a^2-2b)/(-4a))+O(t^{3/2}).
\end{eqnarray*}
the last expansion holds in any case, except  when $\G$ is tangent to the segment ($a\in \R^+$) and over the segment on a neighborhood of $-1$. In this last case, we should  replace $\frac{1}{\sqrt{-2a}}$ by $i$ instead of $-i$.

Then
\begin{eqnarray*} 
\tilde T'(\G)&=&\frac{1}{\sqrt{-2a}}\frac{1}{\sqrt{t}}+1+O(t^{1/2}),\\
\tilde T''(\G)&=&\frac{1}{\sqrt{-2a}^3}\frac{1}{\sqrt{t}^3}+\frac{C_1}{\sqrt{t}}+O(t^{1/2}).
\end{eqnarray*}
and
\begin{eqnarray*}
\g'_1&=&\frac{a}{\sqrt{-2a}}\frac{1}{\sqrt{t}}+a+O(t^{1/2}),\\
\frac{\g'_1}{|\g_1'|}&=&\frac{a}{|a|}\frac{|\sqrt{-2a}|}{\sqrt{-2a}}+ C_2 \sqrt{t}+O(t),\\
\g''_1&=&\frac{a^2}{\sqrt{-2a}^3}\frac{1}{\sqrt{t}^3}+ C_3\frac{1}{\sqrt{t}}+O(1).
\end{eqnarray*}
Now, we can evaluate $A$:
\begin{eqnarray*}
\g''_1- \frac{\g'_1}{|\g_1'|}<\frac{\g'_1}{|\g_1'|},\g''_1>&=&\frac{1}{\sqrt{t}^3}\Bigl(\frac{a^2}{\sqrt{-2a}^3}-\frac{a}{|a|}\frac{|\sqrt{-2a}|}{\sqrt{-2a}}\langle \frac{a}{|a|}\frac{|\sqrt{-2a}|}{\sqrt{-2a}},\frac{a^2}{\sqrt{-2a}^3}\rangle\Bigl)\\
&&+C_4\frac{1}{t}+O(t^{-1/2}).
\end{eqnarray*}
We can easily see that $\arg (a^2/\sqrt{-2a}^3)=\pm\pi + \arg (a/\sqrt{-2a})$, so 
$$\frac{a^2}{\sqrt{-2a}^3}-\frac{a}{|a|}\frac{|\sqrt{-2a}|}{\sqrt{-2a}}\langle\frac{a}{|a|}\frac{|\sqrt{-2a}|}{\sqrt{-2a}},\frac{a^2}{\sqrt{-2a}^3}\rangle=0$$
and $d f_1/ds=C_5+O(t^{1/2})$, which means that $d f_1/ds$ has a limit as $s\to 0$. This argument holds for $\g_2$, doing the calculations with  $\tilde T_{\inte}(z) = z+\sqrt{z^2-1}$. So $d\g_1/ds$ and $d\g_2/ds$ are $C^1$ on $[0,1]$ and we see that $\tilde\G'$ is  Lipschitz because $\tilde\G=\g_1\cup\g_2$.

\vspace{12pt}

Finally, if $\G$ intersects $[-1,1]$ at one point $x=\G(t_0)$, then $\tilde T(\G)$ is the union of two Jordan curves with a jump : $\tilde T(\G(t_0^{-}))=1/\tilde T(\G(t_0^+))$ (see Picture 2). In this case, $\tilde T_{\inte}(\G)$ is also the union of two Jordan arcs which extend $\tilde T(\G)$, indeed 
$$\tilde T_{\inte}(\G(t_0^+))=1/\tilde T(\G(t_0^+))=\tilde T(\G(t_0^-)).$$

\begin{figure}[ht]

                \begin{minipage}[t]{6cm}
\includegraphics[height=7cm]{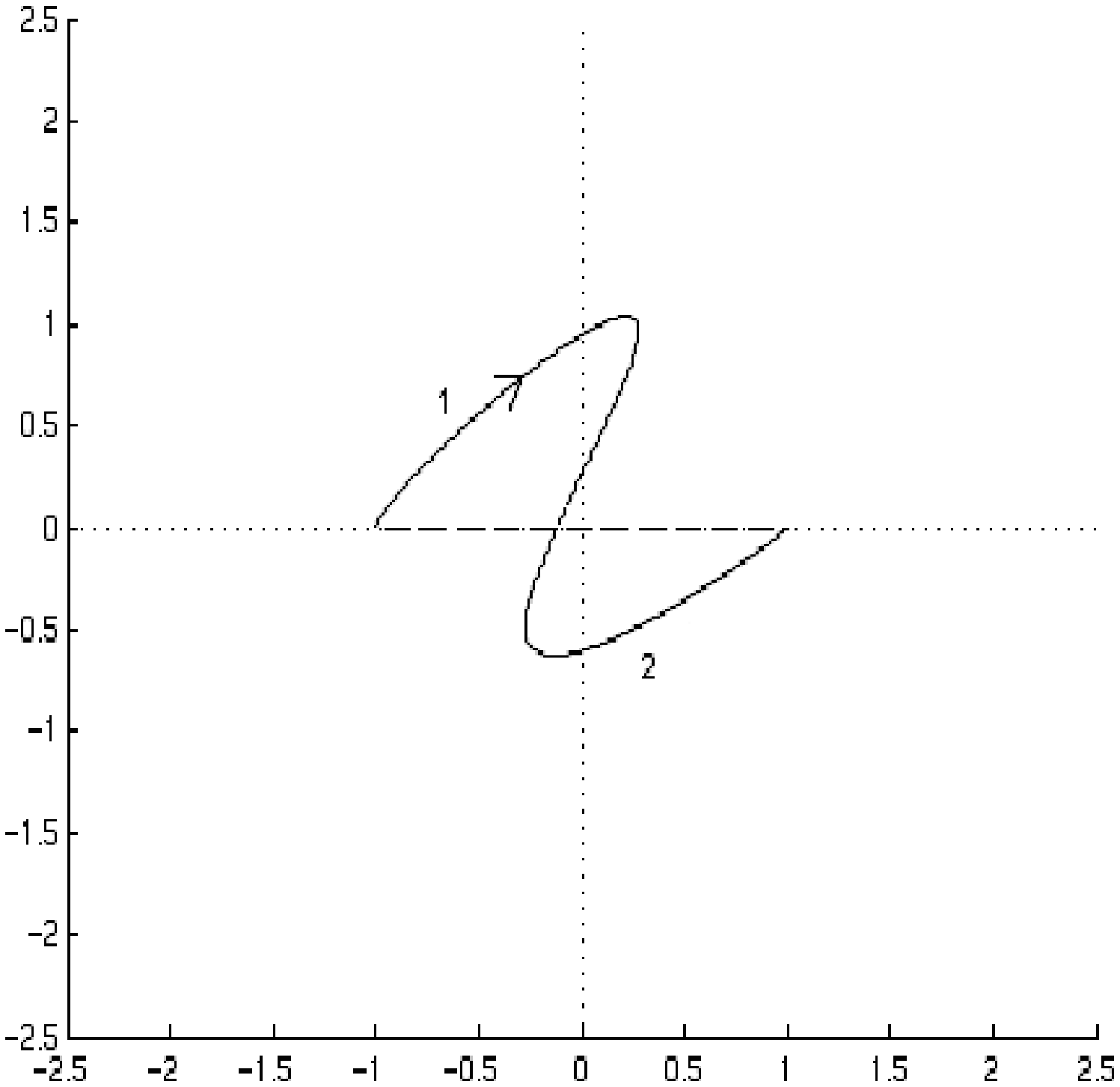}

\centerline{\ \hspace{1.2cm}$\G$}   
                \end{minipage}
                \hfill
                \begin{minipage}[t]{6cm}
\hspace{-1.5cm}
\includegraphics[height=7cm]{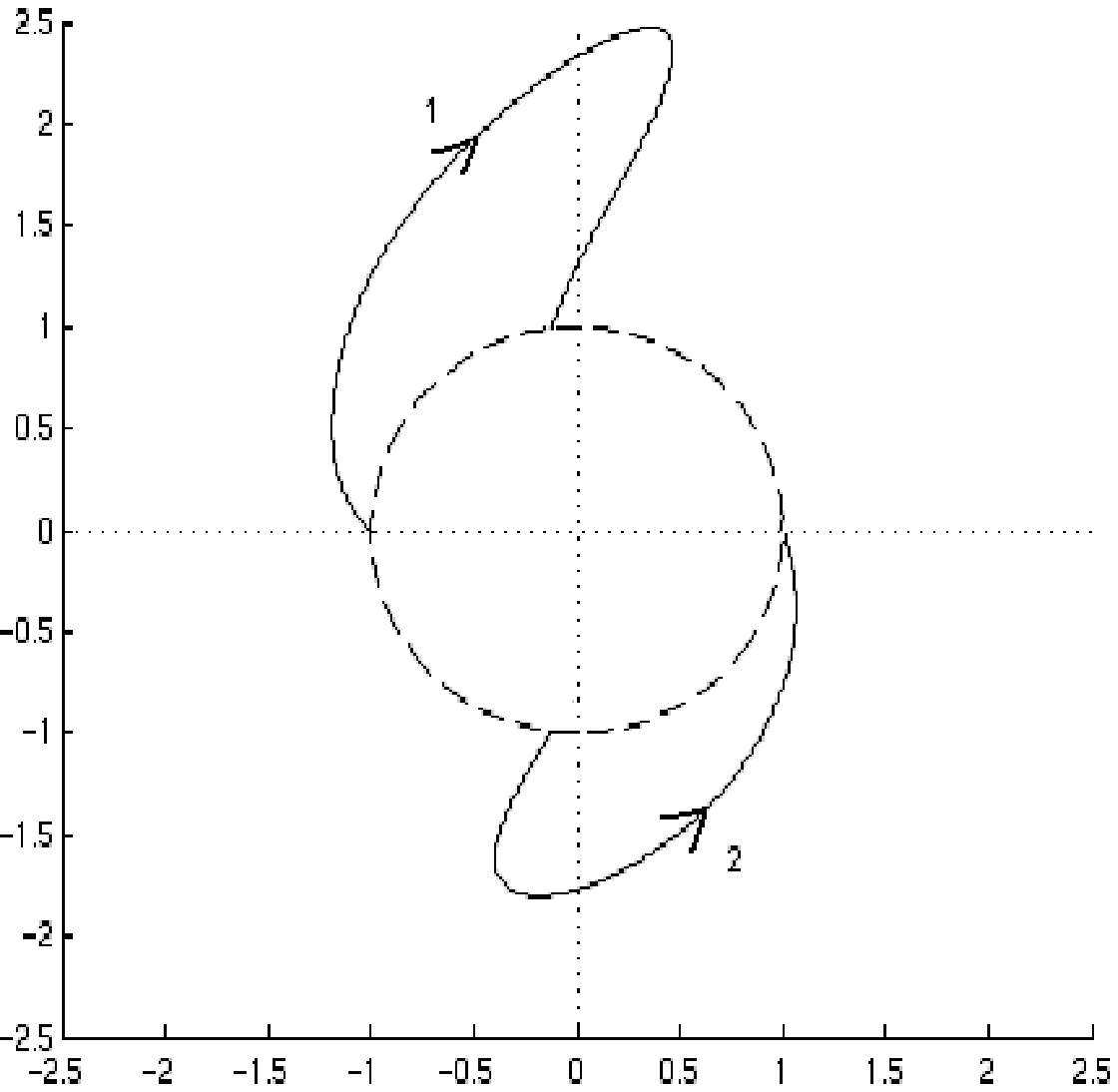}
\centerline{$\tilde T(\G)$\hspace{1cm}\ }
                \end{minipage}
                \end{figure}

\centerline{{\bf PICTURE 2}: $\G$ intersects $[-1,1]$ at one point}

\begin{figure}[ht]
                \begin{minipage}[t]{6cm}
\includegraphics[height=7cm]{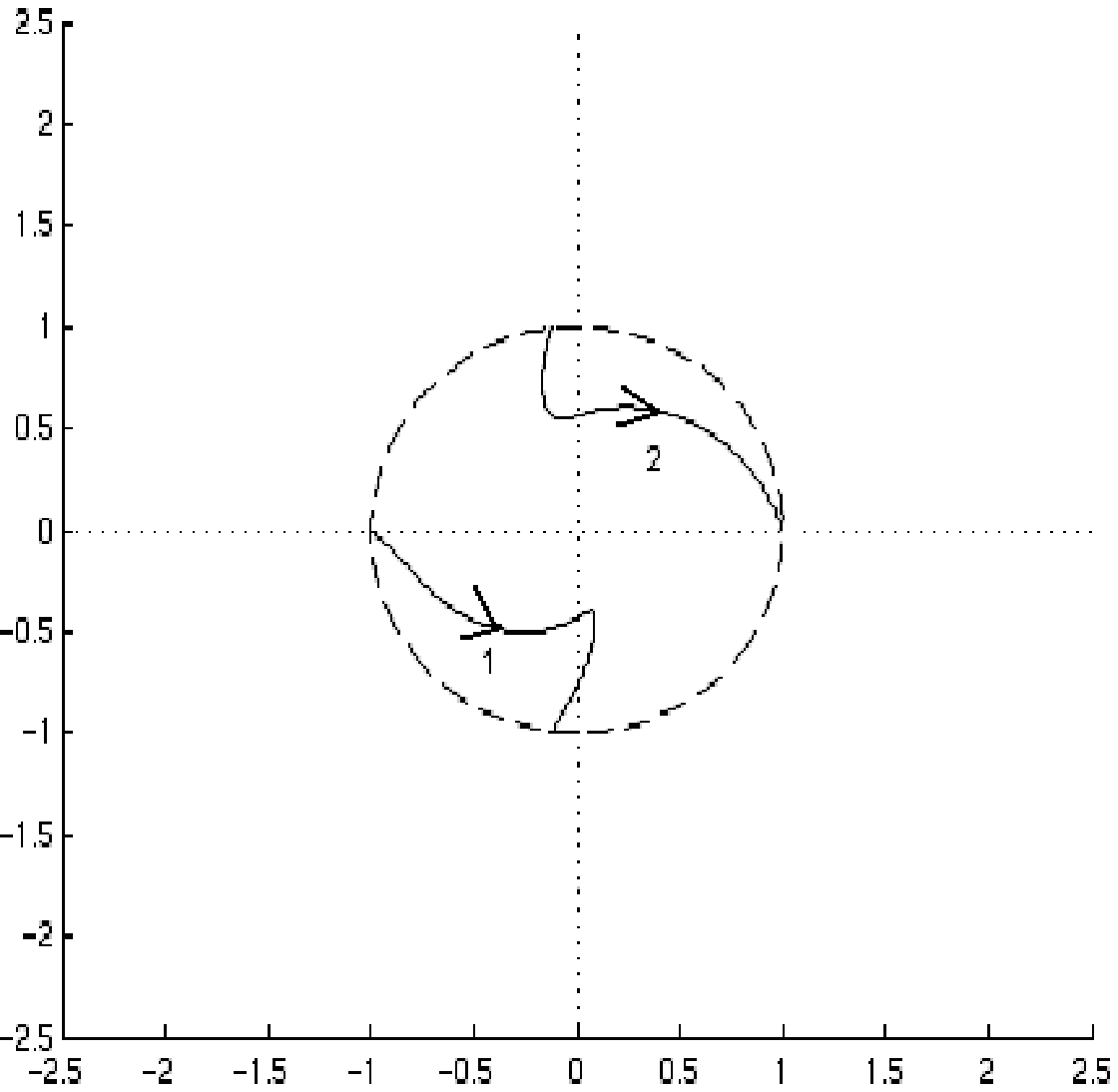}
              
\centerline{\ \hspace{1.3cm}$\tilde T_{\inte}(\G)$}
                \end{minipage}
                \hfill
                \begin{minipage}[t]{6cm}
\hspace{-1.5cm}
\includegraphics[height=7cm]{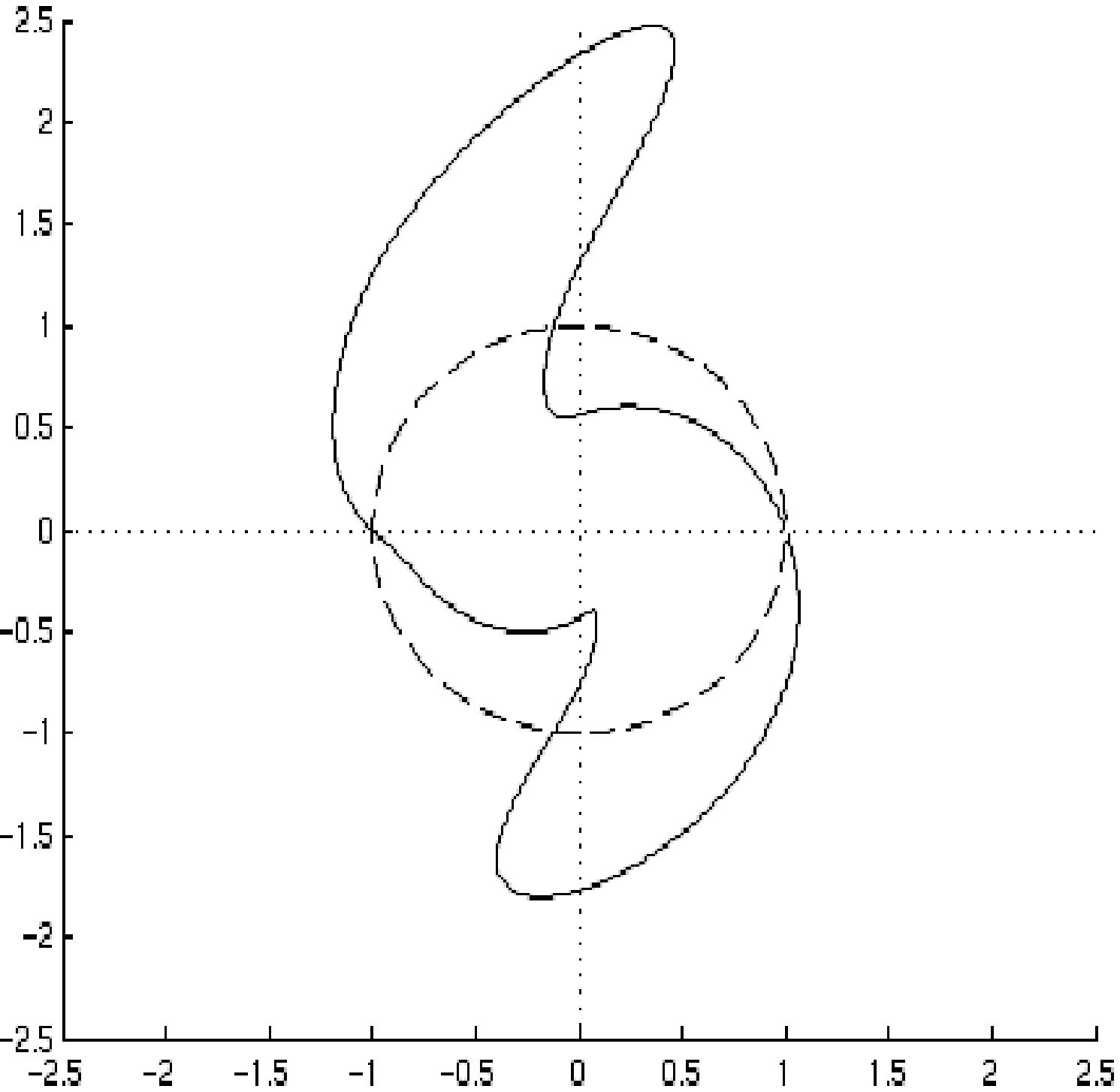}

\centerline{$\tilde\G=G^{-1}(\G)\ \ \ \ \ $}
                \end{minipage}
                \end{figure}

To show the continuity of $\tilde \G'$ on $\tilde T(\G(t_0))$, we consider for example that $x\in (0,1)$ and that $\Im(\G(t_0^-))>0$ and $\Im(\G(t_0^+))<0$. We can compute
\begin{eqnarray*}
\tilde T(\G(t_0^-)) &=&  x+i\sqrt{1-x^2},\\
\tilde T'(\G(t_0^+))&=&  1+xi/\sqrt{1-x^2}, \text{\ because\ }\tilde T=T_+\text{ in a neighborhood of }x\\
\tilde T'(\G(t_0^-))&=&  1-xi/\sqrt{1-x^2},
\end{eqnarray*}
to check that $-\tilde T(\G(t_0^-))^2 \tilde T'(\G(t_0^+))=\tilde T'(\G(t_0^-))$, which allows us to conclude that 
$$\tilde T'_{\inte}(\G(t_0^+))=-1/\tilde T(\G(t_0^+))^2 \tilde T'(\G(t_0^+))=\tilde T'(\G(t_0^-)).$$

We leave to the reader the other cases. Let us do just another case : $x=0$ then 
\begin{eqnarray*}
\tilde T(\G(t_0^-)) &=&  \pm i,\\
\tilde T'(\G(t_0^+))&=&  1, \\
\tilde T'(\G(t_0^-))&=&  1,
\end{eqnarray*}
and as $-(\pm i)^2=1$ we have the continuity of $\tilde\G'$. As $\tilde\G''$ is bounded, $\tilde\G'$ is Lipschitz, so the curve $\tilde\G$ is $C^{1,1}$ and closed. We have just studied the case of one or zero intersection of $\G$ with the segment $(-1,1)$ but this argument works in the general case because we have a finite number of intersection. For example, if $\G\subset [-1,1]$ in a neighborhood of $-1$, then $\tilde G\subset S$, so $\tilde G$ is obviously $C^{1,1}$ in this neighborhood. 

\vspace{24pt}

We denote by $\tilde \Pi$ the unbounded connected component of $\R^2\setminus\tilde\G$. We claim that we can construct $T_2$, a biholomorphism between $\Pi$ and $\tilde \Pi$, such that $T_2^{-1}=G$. Indeed, if we introduce $A=\tilde\Pi\cap \bar D$ and $B=(\inte\ \tilde\Pi^c)\cap D^c$ (see Picture 1), we observe that $B=1/A$, because $\g_2=1/\g_1$ and $1/S=S$. As $G(1/z)=G(z)$, $G$ is bijective on $\inte D^c$ and $1/(\pd D\cap\tilde\Pi)\subset\tilde\Pi^c$ then $G$ is bijective on $\tilde\Pi$ and $G(\tilde\Pi)= \R^2\setminus\G$. Therefore, we have an function $T_2$ mapping the exterior of the Jordan arc to the exterior of the inner domain of a $C^{1,1}$ Jordan curve, such that $T_2^{-1}(z)=1/2(z+1/z)$. 

Next, we just have to use the Riemann mapping Theorem and we find a conformal mapping $F$ between $\tilde\Pi$ and $D^c$, such that $T\equiv F\circ T_2$ maps $\Pi$ to $D^c$ and $F(\infty)=\infty$. \label{F} To finish the proof, we use the Kellogg-Warschawski Theorem (see Theorem 3.6 of \cite{pomm_2}, which can be applied for the exterior problems), to observe that $F$ and $F'$  have a continuous extension up to the boundary. Therefore, adding the fact that $DF$ and $DF^{-1}$ are bounded at infinity (see Remark \ref{2.5}), we find the same properties as in the segment case, in particular that $DT$ blows up at the endpoints of the curve like the inverse of the square root of the distance (see (\ref{T'})).  
\end{proof}

\begin{remark} \label{2.4} If $\G$ intersects the segment $[-1,1]$ infinitely many times, the curve $\tilde \G$ may not be even $C^1$. For example a curve which starts as $t\mapsto (t-1;e^{1/t^2}\sin(1/t)), t\in[0,1/4]$ has two sequences $t_n\to 0$ and $\tilde t_n\to 0$ such $\tilde T'(\G)/|\tilde T'(\G)|$ tends to $i$ and $-i$.
\end{remark}


\begin{remark} \label{2.5}If we have a biholomorphism $H$ between the exterior of an open connected and simply connected domain $A$ and $D^c$, such that $H(\infty)=\infty$, then there exists a nonzero real number $\b$ and a  holomorphic function $h:\Pi\to \C$ such that:
$$H(z)=\b z+ h(z).$$
with 
$$h'(z)=O\Bigl(\frac{1}{|z|^2}\Bigl), \text{ as }|z|\to\infty.$$

This property can be applied for the $F$ above to see that $DF$ and $DF^{-1}$ are bounded.
\end{remark}
\medskip \hspace{-6mm} {\it Proof of Remark \ref{2.5}.}  Indeed, after a translation we can suppose that $0\in \inte\ A$, and we consider $W(z)=1/H(1/z)$. The function $W$ is holomorphic in a neighborhood of $0$ and can be written as $W(z)=a_0+a_1 z+a_2 z^2+...\ $. We have $W(0)=0$ so $a_0=0$. Now we want to prove that $a_1\neq 0$ thanks to the univalence. Indeed, if $a_1=0$, we consider the first non-zero $a_k$, and we observe that there exists $R<0$ such that $|W(z)-a_k z^k|\leq  |a_k||z^k|$ in $B(0,R)$. Next, we denote by $g(z)=a_k z^k$. On $\pd B(0,R)$,  $|W(z)-g(z)|\leq  |a_k|R^k \leq |g(z)|$. Then we can apply the Rouch\'e Theorem to conclude that $W$ and $g$ have the same number of zeros in $B(0,R)$, which is a contradiction with the fact that $W$ is a biholomorphism and $g$ not. Therefore $a_1\neq 0$ and $H(z)=z/a_1+b_0+b_1/z+...$, which ends the proof. Multiplying $H$ by $|a_1|/\bar{a_1}$, we can assume that $\b=1/a_1$ is  real number.
{\centerline{\hfill $\square$}\medskip}

\subsection{The Biot-Savart law.}\label{biot} \ 

Let $\O_0$ be a bounded, open, connected, simply connected subset of the plane, the boundary of which, denoted by $\G_0$, is a $C^\infty$ Jordan curve. We will denote by $\Pi_0$ the unbounded connected component of $\R^2\setminus\G_0$, so that $\O_0^c=\overline{\Pi_0}$.

We denote by $G_{\Pi_0}=G_{\Pi_0} (x,y)$ the Green's function, whose the value is:
\begin{equation}
\label{green}
G_{\Pi_0}(x,y)=\frac{1}{2\pi}\log\frac{|T_0(x)-T_0(y)|}{|T_0(x)-T_0(y)^*||T_0(y)|}
\end{equation}

writing $x^*=\frac{x}{|x|^2}$. The Green's function is the unique function which verifies: 
\begin{equation}
\label{green_equa}
\left\lbrace \begin{aligned}
&\D_y G_{\Pi_0}(x,y)=\d(y-x) \text{ for }x,y\in \Pi_0 \\
&G_{\Pi_0}(x,y)=0 \text{ for }y\in\G_0 \\
&G_{\Pi_0}(x,y)=G_{\Pi_0}(y,x)
\end{aligned}
\right.
\end{equation}

Then the kernel of the Biot-Savart law is $K_{\Pi_0}=K_{\Pi_0}(x,y) \equiv \na_x^\perp G_{\Pi_0}(x,y)$. With $(x_1,x_2)^\perp=\begin{pmatrix} -x_2 \\ x_1\end{pmatrix}$, the explicit formula of $K_{\Pi_0}$ is given by 
\begin{equation}
\label{kernel}
K_{\Pi_0}(x,y)=\frac{((T_0(x)-T_0(y))DT_0(x))^\perp}{2\pi|T_0(x)-T_0(y)|^2}-\frac{((T_0(x)-T_0(y)^*)DT_0(x))^\perp}{2\pi|T_0(x)-T_0(y)^*|^2}.
\end{equation}

We require information on far-field behavior of $K_{\Pi_0}$. We will use several times the following general relation:
\begin{equation}
\label{frac}
\Bigl| \frac{a}{|a|^2}-\frac{b}{|b|^2}\Bigl|=\frac{|a-b|}{|a||b|},
\end{equation}
which can be easily checked by squaring both sides.

We now find the following inequality:
$$|K_{\Pi_0}(x,y)|\leq C \frac{|T_0(y)-T_0(y)^*|}{|T_0(x)-T_0(y)||T_0(x)-T_0(y)^*|}.$$

For $f\in C_c^\infty({\Pi_0})$, we introduce the notation
$$K_{\Pi_0}[f]=K_{\Pi_0}[f](x)\equiv\int_{\Pi_0} K_{\Pi_0}(x,y)f(y)dy.$$

It is easy to see, for large $|x|$, that $\hbox{$|K_{\Pi_0}[f]|(x)\leq \dfrac{C_1}{|x|^2}$}$ where $C_1$ depends on the size of the support of $f$. We used here the explicit formulas for the biholomorphism $T_0$ (Remark \ref{2.5}).

\begin{lemma}\label{2.6} The vector field $u=K_{\Pi_0}[f]$ is a solution of the elliptic system:
\begin{equation*}
\left\lbrace\begin{aligned}
\diver u &=0 &\text{ in } {\Pi_0} \\
\curl u &=f &\text{ in } {\Pi_0} \\
u.\hat{n}&=0 &\text{ on } \G_0 \\
\lim_{|x|\to\infty}|u|&=0
\end{aligned}\right.
\end{equation*}
\end{lemma}

The proof of this Lemma is straightforward.

\subsection{Harmonic vector fields}\ 

We will denote by $\hat{n}$ the unit normal exterior to $\Pi_0$ at $\G_0$. In what follows all contour integrals are taken in the counter-clockwise sense, so that $\int_{\G_0} F.\dss=-\int_{\G_0} F.\hat{n}^\perp ds$.

\begin{proposition}\label{2.7} There exists a unique classical solution $H=H_{\Pi_0}$ of the problem:
\begin{equation}
\left\lbrace\begin{aligned}
\diver H &=0 &\text{ in } {\Pi_0} \\
\curl H &=0 &\text{ in } {\Pi_0} \\
H.\hat{n}&=0 &\text{ on } \G_0 \\
\lim_{|x|\to\infty}|H|&=0 \\
\int_{\G_0} H.\dss &=1
\end{aligned}\right.
\end{equation}

Moreover, $H_{\Pi_0}=\mathcal{O}(1/|x|)$ as $|x|\to\infty$
\end{proposition}

To prove this, one can check that $H_{\Pi_0} (x)=\frac{1}{2\pi}\na^\perp \log |T_0(x)|$ is the unique solution. The details can be found in \cite{ift_lop}.

\section{Flow in an exterior domain}

Let us formulate precisely here the small obstacle limit.

\subsection{The initial-boundary value problem} \

Let $u=u(x,t)=(u_1(x_1,x_2,t),u_2(x_1,x_2,t))$ be the velocity of an incompressible, ideal flow in $\O_0^c$. We assume that $u$ is tangent to $\G_0$ and $u\to 0$ as $|x|\to \infty$. The evolution of such a flow is governed by the Euler equations:
\begin{equation}
\left\lbrace\begin{aligned}
&\pd_t u+u.\na u=-\na p &\text{ in }{\Pi_0}\times(0,\infty) \\
&\diver u =0 &\text{ in } {\Pi_0}\times[0,\infty) \\
&u.\hat{n} =0 &\text{ in } \G_0\times[0,\infty) \\
&\lim_{|x|\to\infty}|u|=0 & \text{ for }t\in[0,\infty)\\
&u(x,0)=u_0(x) &\text{ in } \O^c
\end{aligned}\right.
\end{equation}
 where $p=p(x,t)$ is the pressure. An important quantity for the study of this problem is the vorticity:
$$\o=\curl(u)=\pd_1 u_2 - \pd_2 u_1.$$

The velocity and the vorticity are coupled by the elliptic system:
\begin{equation*}
\left\lbrace\begin{aligned}
&\diver u=0 &\text{ in }{\Pi_0}\times[0,\infty) \\
&\curl u =\o &\text{ in } {\Pi_0}\times[0,\infty) \\
&u.\hat{n} =0 &\text{ in } \G_0\times[0,\infty) \\
&\lim_{|x|\to\infty}|u|=0 & \text{ for }t\in[0,\infty)\\
\end{aligned}\right.
\end{equation*}

Lemma \ref{2.6} and Proposition \ref{2.7} assure us that the general solution of this system is given by $u=u(x,t)=K_{\Pi_0}[\o(.,t)](x)+\a H_{\Pi_0}(x)$ for a function $\a=\a(t)$. However, using the fact that the circulation $\g$ of $u$ around $\G$ is conserved, we prove that $\a$ does not depend on the time, and $\a(t)=\g+\int_{\Pi_0} \curl u_0(x)$ (see \cite{ift_lop}). Therefore, if we give the circulation, then we have the uniqueness of the solution of the previous system.

Finally, we can now write the vorticity formulation of this problem as:
\begin{equation}
\left\lbrace\begin{aligned}
&\pd_t \o+u.\na \o=0 &\text{ in }{\Pi_0}\times(0,\infty) \\
& u=K_{\Pi_0}[\o]+\a H_{\Pi_0}&\text{ in } {\Pi_0}\times[0,\infty) \\
&\o(x,0)=\curl u_0(x) &\text{ in } {\Pi_0}
\end{aligned}\right.
\end{equation}

\subsection{The evanescent obstacle}\label{evanescent}\ 

We will formulate in this subsection a family of problems, parametrized by the size of the obstacle. Therefore, we fix $\o_0$ such that its support is compact and does not intersect $\G$. 

We will consider a family of domain $\O_\e$, containing $\G$, with $\e$ small enough, such that the support of $\o_0$ does not intersect $\O_\e$. If we denote by $T_\e$ the biholomorphism between $\Pi_\e\equiv \O_\e^c$ and $D^c$, then we suppose the following properties:
\begin{assumption}\label{3.1}
The biholomorphism family   $\{T_\e\}$ verifies
\begin{itemize}
\item[(i)] $\|T_\e - T\|_{L^\infty (B(0,R)\cap \Pi_\e)} \to 0$ for any $R>0$,
\item[(ii)] $\det(DT_\e^{-1})$ is bounded in $\Pi_\e$ independently of $\e$,
\item[(iii)] $\| DT_\e- DT\|_{L^3(B(0,R)\cap \Pi_\e)} \to 0$  for any $R>0$,
\item[(iv)] there exist $R>0$ and $C>0$ such that $|DT_\e(x)|\leq C |x|$ on $B(0,R)^c$.
\end{itemize}
\end{assumption}

\begin{remark}\label{3.2} 
We can observe that point (iii) implies that for any $R$, $DT_\e$ is bounded in $L^p(B(0,R)\cap \Pi_\e)$ independently of $\e$, for $p\leq3$.
\end{remark}

Just before going on, we give here one example of an obstacle family.
\begin{example}
We consider $\O_\e\equiv T^{-1}(B(0,1+\e)\setminus D)$. In this case, $T_\e=\frac{1}{1+\e}T$ verifies the previous assumption. If $\G$ is a segment, then $\O_\e$ is the interior of an ellipse around the segment.
\end{example}
The problem of this example is that the shape of the obstacle is the same of $\G$. 

We naturally denote by $\G_\e=\pd\O_\e$ and $\Pi_\e=\inte\ \O_\e^c$. We denote also by $G_\e$,$K_\e$ and $H_\e$ the previous functions corresponding at $\Pi_\e$.

Consider also the following problem :
\begin{equation*}\label{E_e}
 \left\lbrace\begin{aligned}
&\pd_t \o^\e+u^\e.\na \o^\e=0 &\text{ in }\Pi_\e\times(0,\infty) \\
& u^\e=K^\e[\o^\e]+\a H^\e &\text{ in } \Pi_\e\times[0,\infty) \\
&\o^\e(x,0)=\o_0(x) &\text{ in } \Pi_\e
\end{aligned}\right.
\end{equation*}

It follows from the work of Kikuchi \cite{kiku} that, for any $\e>0$, if $\o_0$ is sufficiently smooth then this system has a unique solution.

We now write the explicit formulas for $K^\e$ and $H^\e$:
\begin{equation} \label{K} 
K^\e  = \dfrac{1}{2\pi} DT_\e^t(x)\Bigl(\dfrac{(T_\e(x)-T_\e(y))^\perp}{|T_\e(x)-T_\e(y)|^2}-\dfrac{(T_\e(x)- T_\e(y)^*)^\perp}{|T_\e(x)- T_\e(y)^*|^2}\Bigl)
\end{equation}
and 
\begin{equation}\label{H}
H^\e=\frac{1}{2\pi}DT_\e^t(x)\Bigl(\frac{(T_\e(x))^\perp}{|T_\e(x)|^2}\Bigl).
\end{equation}
We introduce in the same way, $K$ and $H$, replacing $T_\e$ by $T$.

The regularity of $T$ implies that $H^\e$ is bounded in $L^2_{\loc}$, which is really better than the punctual limit for the obstacle (see \cite{ift_lop}) where $H^\e$ is just $L^1_{\loc}$. In our case, the limit is easier to see when $T_\e\to T$, and this extra regularity will allow us the passing to the limit.

\section{{\it A priori} estimates}

These estimates are important to conclude on the asymptotic behavior of the sequences $(u^\e)$ and $(\o^\e)$. The transport nature of (\ref{E_e}) gives us the classical estimates for the vorticity: $\|\o^\e(.,t)\|_{L^p(\Pi_\e)}=\|\o_0\|_{L^p(\R^2)}$  and for $p\in[1,+\infty]$. In this article, we suppose that $\o_0$ is $L^\infty$ and compactly supported. Moreover we choose $\e$ small enough, so that the support of $\o_0$ does not intersect $\Pi_\e$.

\subsection{Velocity estimate}\ 

We begin by recalling a result found in \cite{ift}.
\begin{lemma}\label{4.1} Let $a\in (0,2)$, $S\subset\R^2$ and $h:S\to\R_+$ be a function in $L^1(S)\cap L^\infty(S)$. Then
$$\int_S \frac{h(y)}{|x-y|^a}dy\leq C\|h\|^{1-a/2}_{L^1(S)}\|h\|^{a/2}_{L^\infty(S)}.$$
\end{lemma}

The goal of this subsection is to find a velocity estimate thanks to the explicit formula of $u^\e$ in function of $\o^\e$ and $\g$ (Subsection \ref{evanescent}):
\begin{equation}
\label{u_e}
u^\e(x,t)=\dfrac{1}{2\pi}DT_\e^t(x)(I_1+I_2)+\a H^\e(x)
\end{equation} 
with
\begin{equation}\label{I_1}
I_1= \int_{\Pi_\e}\dfrac{(T_\e(x)-T_\e(y))^\perp}{|T_\e(x)-T_\e(y)|^2}\o^\e(y,t)dy,
\end{equation}
and
\begin{equation}\label{I_2}
I_2=-\int_{\Pi_\e}\dfrac{(T_\e(x)- T_\e(y)^*)^\perp}{|T_\e(x)- T_\e(y)^*|^2}\o^\e(y,t)dy.
\end{equation}

We begin by estimating $I_1$ and $I_2$.

\begin{lemma}\label{I_est} Let $a\in (0,2)$ and $h:\Pi_\e\to\R_+$ be a function in $L^1(\Pi_\e)\cap L^\infty(\Pi_\e)$. We introduce
\begin{equation}\label{I_1a} 
I_{1,a}= \int_{\Pi_\e}\dfrac{|h(y)|}{|T_\e(x)-T_\e(y)|^a}dy
\end{equation}
\begin{equation}\label{I_2a} 
\tilde I_{2}= \int_{\Pi_\e}\dfrac{|h(y)|}{|T_\e(x)- T_\e(y)^*|}dy.
\end{equation}
There exists a constant $C>0$ depending only on the shape of $\G$, such that
$$|I_{1,a}|\leq C \|h\|_{L^1}^{1-a/2} \|h\|_{L^\infty}^{a/2} \text{\ and\ } |\tilde I_2|\leq C (\|h\|_{L^1}^{1/2} \|h\|_{L^\infty}^{1/2}+  \|h\|_{L^1}).$$
\end{lemma}

\begin{remark}\label{remark4.3}
It will be clear from the proof that similar estimates hold true with $T_\e$ replaced by $T$.
\end{remark}

\begin{proof}
We start with the $I_{1,a}$ estimate. Let $J=J(\x)\equiv |\det(DT_\e^{-1})(\x)|$ and $z= T_\e(x)$. Making also the change of variables $\y= T_\e(y)$, we find
\begin{equation}
\label{I_1d}
|I_{1,a}|\leq \int_{|\y|\geq 1}\dfrac{1}{|z-\y|^a}|h( T_\e^{-1}(\y))|J(\y)d\y.
\end{equation}

Next, we introduce $f^\e(\y)=|h( T_\e^{-1}(\y))|J(\y)\h_{\{|\y|\geq 1\}}$, with $\h_E$ the characteristic function of the set $E$. Changing variables back, we get 
$$\|f^\e\|_{L^1(\R^2)}=\|h\|_{L^1}.$$
The second point of Assumption \ref{3.1} allows us to write
$$\|f^\e\|_{L^\infty(\R^2)}\leq C\|h\|_{L^\infty}.$$

So we apply the previous Lemma for $f$ and we finally find
\begin{equation}\label{1_a}
|I_{1,a}|\leq \int_{\R^2}\frac{1}{|z-\y|^a}f^\e(\y)d\y\leq C_1\|f^\e\|_{L^1}^{1-a/2}\|f^\e\|_{L^\infty}^{a/2}.
\end{equation}
This concludes the estimate for $I_{1,a}$.

Let us estimate $\tilde I_2$:
\begin{equation*}
|\tilde I_2|\leq \int_{\Pi_\e}\dfrac{1}{|T_\e(x)- T_\e(y)^*|}|h(y)|dy.
\end{equation*}

We use, as before, the notations $J$, $z$ and the change of variables $\y$
\begin{equation}
\label{I_2d}
|\tilde I_2| \leq  \int_{|\y|\geq 1}\dfrac{1}{|z- \y^*|}|h(T_\e^{-1}(\y))|J(\y)d\y.
\end{equation}

Next, we again change variables writing $\th=\y^*$, to obtain:
\begin{eqnarray*}
|\tilde I_2| &\leq&  \int_{|\th|\leq 1}\frac{1}{|z-\th|}|h( T_\e^{-1}(\th^*))|J(\th^*)\frac{d\th}{|\th|^4} \\
&\leq&  \Bigl(\int_{|\th|\leq 1/2}+\int_{1/2 \leq|\th|\leq 1}\Bigl)\equiv I_{21}+I_{22}
\end{eqnarray*}

First we estimate $I_{21}$. As $z= T_\e(x)$, one has that $|z|\geq 1$, and if $|\th|\leq 1/2$ then $|z-\th|\geq 1/2$. Hence 
\begin{eqnarray}
|I_{21}| &\leq&  \int_{|\th|\leq 1/2}2 |h( T_\e^{-1}(\th^*))|J(\th^*)\frac{d\th}{|\th|^4} \\
&\leq& 2 \int_{|\y|\geq 2 }|h( T_\e^{-1}(\y))|J(\y)d\y\leq 2 \|h\|_{L^1}\label{I_21}
\end{eqnarray}

Finally, we estimate $I_{22}$. Let $g^\e(\th)=|h( T_\e^{-1}(\th^*))|\frac{J(\th^*)}{|\th|^4}$. We have
$$I_{22}=\int_{1/2 \leq|\th|\leq 1}\frac{1}{|z-\th|}g^\e(\th)d\th .$$
As above, we deduce by changing variables back that
$$\|g^\e\|_{L^1(1/2\leq|\th|\leq 1)}\leq \|h\|_{L^1}.$$
It is also trivial to see that
$$\|g^\e\|_{L^\infty(1/2\leq|\th|\leq 1)}\leq C \|h\|_{L^\infty}.$$
By Lemma \ref{4.1}
\begin{eqnarray}
|I_{22}|&=& \int_{1/2\leq|\th|\leq 1} \frac{1}{|z-\th|}g^\e(\th)d\th \\
&\leq& C \|g^\e\|_{L^1(1/2\leq|\th|\leq 1)}^{1/2} \|g^\e\|_{L^\infty(1/2\leq|\th|\leq 1)}^{1/2} \leq C  \|h\|_{L^1}^{1/2}\|h\|_{L^\infty}^{1/2}.\label{I_22}
\end{eqnarray}
\end{proof}

Since $|T_\e(x)|\geq 1$, one can easily see from (\ref{H}) that $|H^\e(x)|\leq |DT_\e(x)|$. Moreover, applying the previous Lemma with $a=1$ and $h=\o^\e \in L^1\cap L^\infty$, we get that
$$|I_{1}|\leq C \|\o^\e\|_{L^1}^{1/2} \|\o^\e\|_{L^\infty}^{1/2} \text{\ and\ } | I_2|\leq C (\|\o^\e\|_{L^1}^{1/2} \|\o^\e\|_{L^\infty}^{1/2}+  \|\o^\e\|_{L^1}).$$
Thanks to the explicit formula (\ref{u_e}), we can deduce directly the following Theorem:

\begin{theorem}\label{4.2} $u^\e$ is bounded in $L^\infty(\R^+,L^2_{\loc}(\Pi_\e))$ independently  of $\e$. More precisely, there exists a constant $C>0$ depending only on the shape of $\G$ and the initial conditions $\|\o_0\|_{L^1}$, $\|\o_0\|_{L^\infty}$, such that 
$$\|u^\e(.,t)\|_{L^p(S)}\leq C\|DT_\e\|_{L^p(S)}, \text{\ for all\ } p \in [1,\infty] \text{\ and for any subset } S \text{\ of }\Pi_\e. $$
\end{theorem}

The difference with \cite{ift_lop} is that we have an estimate $L^p_{\loc}$ instead of $L^\infty$, but in our case, this estimate concerns all the velocity $u^\e$. It is one of the reason of the use of a different method to the velocity convergence.

\subsection{Cutoff function}\label{cutoff}\

If we want to compare the different velocity and vorticity, the issue is that $u^\e$ and $\o^\e$ are defined on an $\e$-dependent domain. For this reason we extend the velocity and vorticity on $\R^2$ by multiplying by an $\e$-dependent cutoff function for a neighborhood of $\O_\e$. 

Let $\F\in C^\infty(\R)$ be a non-decreasing function such that $0\leq\F\leq 1$, $\F(s)=1$ if $s\geq 2$ and $\F(s)=0$ if $s\leq 1$. Then we introduce 
$$\F^\e=\F^\e(x)=\F\Bigl(\frac{|T_\e(x)|-1}{\e}\Bigl).$$
Clearly $\F^\e$ is $C^\infty(\R^2)$ vanishing in a neighborhood of $\overline{\O^\e}$.

We require some properties of $\na\F^\e$ which we collect in the following Lemma.

\begin{lemma}\label{4.4} The function $\F^\e$ defined above has the following properties:
\begin{itemize}
\item[(a)] $H^\e.\na\F^\e\equiv 0$ in $\Pi_\e$,
\item[(b)] there exists a constant $C>0$ such that the Lebesgue measure of the support of $\F^\e-1$ is bounded by $C\e$.
\end{itemize}
\end{lemma}

\begin{proof}
First, we remark that
$$H^\e(x)=\frac{1}{2\pi}\na^\perp \log |T_\e(x)|=\frac{1}{2\pi|T_\e(x)|}\na^\perp |T_\e(x)|,$$
and
\begin{equation}\label{Fi}
\na \F^\e=\frac{1}{\e}\F'\Bigl(\frac{|T_\e(x)|-1}{\e}\Bigl)\na |T_\e(x)|
\end{equation}
what gives us the first point.

Finally, the support of $\F^\e -1$ is contained in the subset $\{ x\in\Pi_\e | 1  \leq |T_\e(x)| \leq 1 + 2\e\}$. The Lebesgue measure can be estimated as follows:
$$\int_{ 1  \leq |T_\e(x)| \leq 1 + 2\e}dx=\int_{1  \leq |z| \leq 1 + 2\e} |\det(DT_\e^{-1})|(z) dz \leq C_1 \e$$
for $\e$ small enough.
\end{proof}

We introduced the cutoff function $\F^\e$ in order to extend the velocity and the vorticity to $\mathbb{R}^2$. One needs to make sure that the limit velocity and vorticity are not affected by the way the extension is constructed. We observe that our method of extension does not produce an error in the limit velocity and vorticity. Indeed, we denote by $ \tilde u^\e$, respectively $\tilde\o^\e$, the extension of $u^\e$, respectively $\o^\e$, by 0 inside the obstacle and we prove that $\displaystyle\lim_{\e\to 0} \tilde u^\e = \lim_{\e\to 0} \F^\e u^\e$ and $\displaystyle\lim_{\e\to 0} \tilde\o^\e = \lim_{\e\to 0} \F^\e \o^\e$ in $D'(\R^2)$. Indeed, using Theorem \ref{4.2} and Remark \ref{3.2}, point (b) of the previous Lemma allows us to state that 
$$\|\F^\e u^\e - u^\e\|_{L^2(\Pi_\e)} \leq C \|DT_\e\|_{L^3(\supp(\F^\e-1))} |C\e|^{1/6}$$
and 
$$\|\F^\e \o^\e - \o^\e\|_{L^p(\Pi_\e)} \leq C \|\o_0\|_{L^\infty} |C\e|^{1/p}$$
for all $p\in [1,\infty)$.

In the case where the limit is a point (\cite{ift_lop}), the Lebesgue measure of the support of $\na \F^\e$ is bounded by $C\e^2$, which implies that the norm $L^1$ of this gradient tends to $0$. Moreover, the authors use a part of velocity $v^\e$ bounded independently of $\e$, so they can compute the limit of  $v^\e . \na \F^\e$ and  $v^\e . \na^\perp \F^\e$ which is necessary for the calculation of the curl and div. Finally they conclude thanks to the Div-Curl Lemma. 

In our case, for $1\leq p< 4$ we have $\|\na\F^\e\|_{L^p}\leq C_p/\e$, and we can not compute the limit of $u^\e .\na^\perp\F^\e$. For this reason we can not use a similar proof as in \cite{ift_lop}.

However, the following Lemma gives us a piece of information about the limit behavior.

\begin{lemma}\label{4.5} $u^\e .\na\F^\e \to 0$ strongly in $L^1(\R^2)$ and uniformly in time, when $\e\to 0$.
\end{lemma}

\begin{proof}
Using the explicit formulas (\ref{u_e}), (\ref{I_1}), (\ref{I_2}) and (\ref{Fi}), we write 

\begin{eqnarray*}
 u^\e(x).\na\F^\e(x) & = &  u^{\e\perp}(x).\na^\perp\F^\e(x) \\ 
& = &  -\frac{1}{2\pi\e} \F'\Bigl(\frac{|T_\e(x)|-1}{\e} \Bigl) \int_{\Pi_\e}\Bigl(\dfrac{T_\e(x)-T_\e(y)}{|T_\e(x)-T_\e(y)|^2}-\dfrac{T_\e(x)- T_\e(y)^*}{|T_\e(x)- T_\e(y)^*|^2}\Bigl)\o^\e(y,t)dy \\
&& \times DT_\e(x)DT_\e^t(x)\frac{T_\e(x)^\perp}{|T_\e(x)|}.
\end{eqnarray*}

But $T_\e$ is  holomorphic, so $DT_\e$ is of the form 
$\begin{pmatrix}
a & b \\
-b & a
\end{pmatrix}$
 and we can check that $DT_\e(x)DT_\e^t(x)=(a^2+b^2)Id=|\det(DT_\e)(x)|Id$, so
\begin{eqnarray*}
u^\e(x).\na\F^\e(x) & = & \frac{\F'(\frac{|T_\e(x)|-1}{\e})|\det(DT_\e)(x)|}{2\pi\e|T_\e(x)|} \\
&&\times \int_{\Pi_\e}\Bigl(\dfrac{T_\e(y).T_\e(x)^\perp}{|T_\e(x)-T_\e(y)|^2}-\dfrac{ T_\e(y)^*.T_\e(x)^\perp}{|T_\e(x)- T_\e(y)^*|^2}\Bigl)\o^\e(y,t)dy.
\end{eqnarray*}

We compute the $L^1$ norm, next  we change variables twice $\y=T_\e(y)$ and $z=T_\e(x)$, to have

\begin{eqnarray*}
\|u^\e.\na\F^\e\|_{L^1} & \\
=\frac{1}{2\pi\e} \int_{|z|\geq 1} &\Bigl|\F'\Bigl(\frac{|z|-1}{\e}\Bigl)\Bigl| \Bigl| \displaystyle\int_{|\y|\geq 1} \Bigl(\dfrac{\y.z^\perp/|z|}{|z-\y|^2}-\dfrac{ \y^*.z^\perp/|z|}{|z-\y^*|^2}\Bigl)J(\y)\o^\e(T_\e^{-1}(\y),t)d\y\Bigl|dz,
\end{eqnarray*}
where $J(\y)= |\det(DT_\e^{-1})(\y)|$.

Thanks to Lemma \ref{4.4}, we know that $\Bigl\|\frac{1}{\e}\F'\Bigl(\frac{|z|-1}{\e}\Bigl)\Bigl\|_{L^1}\leq C$. So it is sufficient to prove that
\begin{equation}
\label{tronc}
\Bigl\| \int_{|\y|\geq 1} \Bigl(\dfrac{\y.z^\perp/|z|}{|z-\y|^2}-\dfrac{ \y^*.z^\perp/|z|}{|z- \y^*|^2}\Bigl)J(\y)\o^\e(T_\e^{-1}(\y),t)d\y\Bigl\|_{L^\infty(1+\e\leq |z|\leq 1+2\e)}\to 0
\end{equation}
as $\e\to 0$, uniformly in time.

Let
$$ A=\dfrac{\y.z^\perp/|z|}{|z-\y|^2}-\dfrac{ \y^*.z^\perp/|z|}{|z- \y^*|^2}.$$
We compute
\begin{eqnarray*}
A&=& \Bigl(\dfrac{(|z|^2-2 z.\y/|\y|^2+1/|\y|^2)-1/|\y|^2(|z|^2-2z.\y+|\y|^2)}{|z-\y|^2|z- \y^*|^2}\Bigl)\y.\frac{z^\perp}{|z|} \\
&=& \frac{(|z|^2-1)(1-1/|\y|^2)}{|z-\y|^2|z- \y^*|^2}\y.\frac{z^\perp}{|z|}.
\end{eqnarray*}

We now use that $|z|\geq 1$, to write
$$|z- \y^*|\geq  1-\frac{1}{|\y|}.$$
Moreover, $| \y^*|\leq 1$ allows to have
$$|z- \y^*|\geq |z|-1.$$

We can now estimate $A$ by:

$$|A|\leq \frac{(|z|+1)(1+1/|\y|)(|z|-1)^b}{ |z-\y|^2|z- \y^*|^b} \Bigl|\y.\frac{z^\perp}{|z|}\Bigl|$$
with $0\leq b \leq 1$, to be chosen later. We remark also that $\y.\frac{z^\perp}{|z|}=(\y-z).\frac{z^\perp}{|z|}$ and the Cauchy-Schwarz inequality gives 
$$\Bigl|\y.\frac{z^\perp}{|z|}\Bigl|\leq |\y-z|.$$

We now use the fact that $|z|-1 \leq 2\e$, to estimate (\ref{tronc}):
$$\Bigl| \int_{|\y|\geq 1}A J(\y)\o^\e(T_\e^{-1}(\y),t)d\y \Bigl|\leq (2+2\e).2.(2\e)^b \int_{|\y|\geq 1} \frac{J(\y)|\o^\e(T_\e^{-1}(\y),t)|}{|z-\y||z- \y^*|^b}d\y.$$

In the same way we passed from (\ref{I_1d}) to (\ref{1_a}), we obtain for $p<2$:

$$\Bigl\|\frac{J(\y)^{1/p}|\o^\e(T_\e^{-1}(\y),t)|^{1/p}}{|z-\y|}\Bigl\|_{L^p}=\Bigl(\int_{|\y|\geq 1} \frac{J(\y)|\o^\e(T_\e^{-1}(\y),t)|}{|z-\y|^p}\Bigl)^{1/p} \leq C_p.$$

Moreover, as we passed from (\ref{I_2d}) to (\ref{I_21}) and (\ref{I_22}), we obtain for $bq=1$:

$$ \Bigl\|\frac{J(\y)^{1/q}|\o^\e(T_\e^{-1}(\y),t)|^{1/q}}{|z- \y^*|^b}\Bigl\|_{L^q}=\Bigl(\int_{|\y|\geq 1} \frac{J(\y)|\o^\e(T_\e^{-1}(\y),t)|}{|z-\y^*|}\Bigl)^{1/q}\leq C_q.$$

We choose $b>0$, $1/p+1/q=1$, and using the H\"older inequality we finish the proof. For example if we fix $b=1/4$, $q=4$ and $p=4/3$, we obtain 
$$\|u^\e.\na\F^\e\|_{L^1} \leq C(2+2\e).2.(2\e)^{1/4}C_{4/3}C_{4}$$
which tends to zero when $\e$ tends to zero.
\end{proof}

If the proof is a little bit technical, the idea is natural. On the boundary, the velocity $u^\e$ is tangent to $\G_\e$, whereas $\na \F^\e$ is normal. To see that, we can check that $A=0$ when $x\in \G_\e$ (which means that $|z|=|T_\e(x)|=1$).

Before going to the last section, we derive directly from the PDE a temporal estimate for the vorticity.

\subsection{Temporal estimate}\

If we fix $\T>0$, we remark that there exists $R_1>0$ such that the support of $\o^\e(.,t)$ is contained in $B(0;R_1)$ for all $0\leq t\leq \T$.

To see that, let $R_0$ be such that $B(0;R_0)$ contains  the support of $\o_0$. Equation $(\ref{E_e})$ means that $\o^\e$ is transported by the velocity field $u^\e$ and the trajectory of a particle moving with the flow verifies 
$$\pd_t X = u^\e(X,t).$$
Moreover, Theorem \ref{4.2} states $|u^\e(x)|\leq C |DT_\e(x)|$ and the last point of Assumption \ref{3.1} states that there exist $R>0$ and $C_1>0$ such that  $DT_\e$ is bounded by $C_1 |x| $ outside $B(0,R)$. If a material particle reaches the region $B(0,\max (R_0,R))^c$, its velocity is uniformly bounded by $CC_1 |x|$, and we obtain the following inequality:
$$\pd_t |X|^2 =X.\pd_t X\leq CC_1 |X|^2,$$
that holds true in such a region. Applying Gronwall Lemma, we observe that the trajectory of a material particle is bounded independently of $\e$ (up to the fixed time $\T$). 

\begin{lemma}\label{4.6} There exists a constant $C$, which does not depend on $t\in [0,\T]$ and $\e$ such that
$$\| \F^\e \pd_t \o^\e\|_{H^{-2}} \leq C.$$
\end{lemma}

\begin{proof} We write the equation verified by  $\F^\e\o^\e$:
\begin{eqnarray*}
\F^\e\pd_t \o^\e &=& -\F^\e u^\e .\na \o^\e \\
&=& -\diver(\F^\e u^\e \o^\e) + \o^\e u^\e .\na \F^\e 
\end{eqnarray*}
which is bounded in $H^{-2}$ for the following reason. Note that $\F^\e$ and $\o^\e$ are uniformly bounded, $u^\e$ is bounded in $L^2(B(0,R_1))$ thanks to Theorem \ref{4.2}, Remark \ref{3.2} and the previous remark. Moreover $ \na \F^\e .u^\e \to 0$ in $L^1$ according to the previous Lemma. We finally conclude, bearing in mind that $L^1$ and $H^{-1}$ are embedded into $H^{-2}$.
\end{proof}

\section{Passing to the limit}

\subsection{Strong compactness in velocity.}\

Fix $\T>0$.
We will need the following Lemmas to the passing to the strong limit $L^2_{\loc}([0,\T]\times \R^2)$ of the sequence $\F^\e u^\e$.

As in the previous subsection, let $R_1>0$ be such that the support of $\o^\e(.,t)$ is contained in $B(0;R_1)$ for all $0\leq t\leq \T$ and $0<\e<\e_0$. 

\begin{lemma}\label{5.1}
For all fixed $x\in \Pi$, there exists $\e_x>0$ such that $x\in \Pi_\e$ for all $\e\leq \e_x$. The two following functions
$$f_{x,\e}(y)=\frac{T_\e(x)-T_\e(y)}{|T_\e(x)-T_\e(y)|^2}$$
$$g_{x,\e}(y)=\frac{T_\e(x)-T_\e(y)^*}{|T_\e(x)- T_\e(y)^*|^2}$$
are bounded in $L^{4/3}(B(0,R_1)\cap \Pi_\e)$ independently of $\e\leq \e_x$ (but not necessarily independent of $x$).

Moreover, 
$$f_{x}(y)=\frac{T(x)-T(y)}{|T (x)-T (y)|^2}$$
and
$$g_{x}(y)=\frac{T (x)-T (y)^*}{|T (x)- T (y)^*|^2}$$
are bounded in $L^{4/3}(B(0,R_1))$.
\end{lemma}

\begin{proof} Bearing in mind the properties of $T$ and that $T_\e\to T$ uniformly in $B(0,R_1)$, we know that $T_\e(B(0,R_1))\subset B(0,\tilde R)$, for some $\tilde R>0$ independent of $\e$.

To bound $f_{x,\e}$, we change variables $\y=T_\e(y)$ and denote by $z=T_\e(x)$:

\begin{eqnarray*}
\int_{B(0,R_1)\cap\Pi_\e}\frac{1}{|T_\e(x)-T_\e(y)|^{4/3}}dy &\leq& \int_{1\leq|\y|\leq \tilde R}\frac{|\det(DT_\e^{-1})|(\y)}{|z-\y|^{4/3}}d\y \\
&\leq& 2\pi C \int_0^{|z|+\tilde R}\frac{1}{r^{1/3}}dr \\
&\leq& C_1(x).
\end{eqnarray*}

For the second function, we begin in the same way, next we change again variables with $\th = \y^*$:

\begin{eqnarray*}
\int_{B(0,R_1)\cap\Pi_\e}\frac{1}{|T_\e(x)-T_\e(y)^*|^{4/3}}dy &\leq & \int_{1\leq|\y|\leq \tilde R}\frac{|\det(DT_\e^{-1})|(\y)}{|z-\y^*|^{4/3}}d\y \\
&\leq& C \int_{1/\tilde R\leq |\th|\leq 1}\frac{1}{|z-\th|^{4/3}}\frac{d\th}{|\th|^4} \\
&\leq& \tilde C \int_{|\th|\leq 1}\frac{1}{|z-\th|^{4/3}}d\th \\
&\leq& C_2(x)
\end{eqnarray*}
with $C_2$ not depending on $\e$.

Replacing $T$ by $T_\e$, one can obtain the bounds for $f_x$ and $g_x$.
\end{proof}

We now consider the limit of $\F^\e \o^\e$. As $\F^\e \o^\e\in L^\infty([0,\infty),L^1\cap L^\infty(\R^2))$ and $\F^\e \pd_t \o^\e\in L^\infty_{\loc}([0,\infty),H^{-2}(\R^2))$, we can extract a subsequence such that $\F^\e \o^\e \to \o$ weak-$*$ in  $L^\infty([0,\infty),L^4(\R^2))$ with $\o\in L^\infty([0,\infty),L^1\cap L^\infty(\R^2))$ and $\pd_t\o\in L^\infty([0,\infty),H^{-2}(\R^2))$. Moreover, we want that $\F^\e \o^\e(.,t)\to \o(.,t)$ weak $L^4$ for all $t$. As $\|\F^\e \o^\e(.,t)\|_{L^4(\R^2)}\leq \|\o_0\|_{L^4}<\infty$, and thanks to Alaoglu's Theorem, for all $t$ we can extract a subsequence which verifies $\F^\e \o^\e(.,t)\to \o(.,t)$ weak $L^4$. The problem is that the subsequence depends on the time. Let us look for a common subsequence for all $t$. We observe that $\omega(t)$ is defined for all $t\geq0$. Indeed, $\pd_t\o\in L^\infty([0,\infty),H^{-2}(\R^2))$, so $\o\in C^0([0,\infty),H^{-2}(\R^2))$. Moreover, since $\o\in L^\infty([0,\infty),L^1\cap L^\infty(\R^2))$ we also have that $\omega(t)\in L^p(\R^2)$ for all $t\geq0$ and $\sup\limits_{t\geq0}\|\omega(t)\|_{L^p(\R^2)}<\infty$ for all  $p>1$.

\begin{proposition}\label{5.2}
There exists a subsequence of $\F^\e \o^\e$ (again denoted by $\F^\e \o^\e$) such that $\F^\e \o^\e(.,t)\to \o(.,t)$ weak   $L^4(\R^2)$ for all $t$.
\end{proposition}

\begin{proof}
We can choose a common subsequence for all rational times, by the diagonal extraction, because $\mathbb{Q}$ is countable. We now prove that this subsequence converges for any $t$.

 Let $\f\in C^\infty_c(\R^2)$, and
$$f_\e(t) = \int_{\R^2} \f(\F^\e\o^\e-\o)dx.$$

Then $f_\e\to 0 $ for all $t\in \mathbb{Q}$. Moreover 
$$ f'_\e(t) = \int_{\R^2} \f(\F^\e\pd_t \o^\e-\pd_t \o)dx,$$
which allows us to state that the family $\{f_\e\}$ is equicontinuous, using the temporal estimate (Lemma \ref{4.6}) and that $\f\in H^2$.

Therefore, we have an equicontinuous family which tends to $0$ on a dense subset, so it tends to $0$ for all $t$.

To finish, let $\f\in L^{4/3}(\R^2)$. The set $C^\infty_c$ being dense into $L^{4/3}$, there exists a sequence $\f_n\in C^\infty_c(\R^2)$ which converges to $\f$ in $L^{4/3}$. We introduce  $f_{n,\e}$ in the same way as $f_\e$, replacing $\f$ by $\f_n$. Let $t$ be fixed, we have by the first part
\begin{equation} \label{fn}
\text{for all }n,\  f_{n,\e} \to 0 \text{ as }\e\to 0.
\end{equation}

Moreover
\begin{eqnarray*}
f_\e-f_{n,\e} &=& \int_{\R^2} (\f-\f_n)(\F^\e\o^\e-\o)dx \\
|f_\e-f_{n,\e}| &\leq& (\|\o_0\|_{L^4}+ \sup_{t\geq0}\|\o(t)\|_{L^4})\|\f-\f_n\|_{L^{4/3}}.
\end{eqnarray*}

Therefore $f_\e-f_{n,\e}$ tends to $0$ uniformly in $\e$, which according to (\ref{fn}) allows to conclude that $f_\e\to 0$. This completes the proof.
\end{proof}

\begin{lemma}\label{5.3} The two following functions:
$$f_\e(x)=\int_{\R^2}\frac{(T(x)-T(y))^\perp}{|T(x)-T(y)|^2}(\F^\e(y)\o^\e(y)-\o(y))dy $$
and
$$g_\e(x)=\int_{\R^2}\frac{(T(x)-T(y)^*)^\perp}{|T(x)-T(y)^*|^2}(\F^\e(y)\o^\e(y)-\o(y))dy $$
tend to $0$ as $\e\to 0$ in $L^2([0,\T],L^6_{\loc}(\R^2))$.
\end{lemma}

\begin{proof} Let $K$ be a compact set of $\R^2$. Firstly, we fix $t\in [0,\T]$ and we prove that the norm $L^6(K)$ of $f_\e$ tends to $0$.

For all $x\in K\setminus \G$, $$f_\e(x)=\int_{\R^2}f_x(y)^\perp (\F^\e(y)\o^\e(y)-\o(y))dy,$$
with $f_x$ given in Lemma \ref{5.1}. Furthermore, Lemma \ref{5.1} states that  $f_x$ is bounded in $L^{4/3}$ and as $\F^\e \o^\e\to \o$ weak   $L^4$, we obtain that for fixed $x$,
$$f_\e(x)\to 0 \text{ as }\e\to 0.$$

Moreover, we can apply Lemma \ref{I_est} (estimate of $I_{1,a}$) to $f_\e$, with $h(y)=\F^\e(y)\o^\e(y)-\o(y)$ and we obtain a bound for $f_\e$ independently of $x$, $t$ and $\e$.

Then,  $f_\e^6 \to 0$ almost everywhere as $\e\to 0$, and $|f_\e^6|$ is uniformly bounded. We can apply the dominated convergence theorem to conclude that for fixed $t$ and $K$ a bounded set
$$\int_K |f_\e(x)|^6 dx \to 0.$$

We now let $t$ vary and we apply again the dominated convergence theorem to obtain the result on $f_\e$.

Using the estimate of $\tilde I_2$ in Lemma \ref{I_est}, we proceed in the same manner to prove the result for $g_\e$.
\end{proof}

Moreover, we need two last Lemmas which are a consequence of the convergence of $T_\e$ to $T$ (Assumption \ref{3.1}).

\begin{lemma}\label{5.4} For all fixed $x\in \Pi$,
$$ \frac{(T_\e(x)-T_\e(y)^*)^\perp}{|T_\e(x)-T_\e(y)^*|^2}-\frac{(T(x)-T(y)^*)^\perp}{|T(x)-T(y)^*|^2}\to 0 \text{\ as\ }\e\to 0$$
in $L^{4/3}(B(0,R_1)\cap\Pi_\e)$ (where the norm is taken with respect to $y$).
\end{lemma}

\begin{proof} Let $x$ be fixed. Using the relation (\ref{frac}) we have
$$\int_{B(0,R_1)\cap\Pi_\e} \Bigl(\frac{|(T(x)-T_\e(x))-(T(y)^*-T_\e(y)^*)|}{|T(x)-T(y)^*||T_\e(x)-T_\e(y)^*|} \Bigl)^{4/3}dy\equiv \int_{B(0,R_1)\cap\Pi_\e}h_{x,\e}(y)dy.$$

By Assumption \ref{3.1}, we know that $h_{x,\e}(y)\to 0$ pointwise as $\e\to 0$. Moreover, if $x\notin \G$ then $|T(x)| > 1$, and as $|T_\e(x)|\to |T(x)|\neq 1$ we can write
\begin{eqnarray*}
|T(x)-T(y)^*| & \geq& |T(x)|-1 >0 \\
|T_\e(x)-T_\e(y)^*| &\geq& |T_\e(x)|-1 \geq  1/2(|T(x)|-1) > 0,
\end{eqnarray*}
for $\e$ small enough (depending on $x$). Then $h_{x,\e}$ can be bounded by a constant which does not depend on $y$ and $\e$, which allows us to apply the dominated convergence theorem to deduce that $\int_{B(0,R_1)\cap\Pi_\e}h_{x,\e}(y)dy\to 0$ as $\e\to 0$.
\end{proof}

\begin{lemma}\label{5.5} One has that
$$ \F^\e(x) \int_{\R^2} \Bigl(\frac{(T_\e(x)-T_\e(y))^\perp}{|T_\e(x)-T_\e(y)|^2}-\frac{(T(x)-T(y))^\perp}{|T(x)-T(y)|^2}\Bigl)\F^\e(y)\o^\e(y)dy\to 0$$
in $L^\infty_{\loc}([0,\infty)\times \R^2)$ as $\e\to 0$.
\end{lemma}

\begin{proof}
Let $x\in B(0,R)$ and $y\in \supp \o^\e$. Using the relation (\ref{frac}) we can introduce and bound
$$\tilde h_{x,\e}\equiv \frac{|(T_\e(x)-T(x))-(T_\e(y)-T(y))|}{|T_\e(x)-T_\e(y)||T(x)-T(y)|} $$
$$\leq \sqrt{2\sup_{B(0,R_2)\cap \Pi_\e}(|T_\e-T|)}\Bigl(\frac{1}{\sqrt{|T_\e(x)-T_\e(y)|}|T(x)-T(y)|}+\frac{1}{|T_\e(x)-T_\e(y)|\sqrt{|T(x)-T(y)|}}\Bigl)$$

where $R_2=max(R,R_1)$.

Using the Lemma \ref{I_est} with $a=5/4$ and $a=5/3$ we conclude by the H\"older inequality:
\begin{eqnarray*}
\int_{\R^2}\tilde h_{x,\e}(y)\F^\e(y)|\o^\e(y)| dy  &\leq &\sqrt{2\sup_{B(0,R_2)\cap \Pi_\e}(|T_\e-T|)} \Bigl( \Bigl\|\frac{(\o^\e)^{2/5}}{\sqrt{|T_\e(x)-T_\e(y)|}}\Bigl\|_{L^{5/2}} \Bigl\|\frac{(\o^\e)^{3/5}}{|T(x)-T(y)|}\Bigl\|_{L^{5/3}}  \\
&\ & + \Bigl\|\frac{(\o^\e)^{3/5}}{|T_\e(x)-T_\e(y)|}\Bigl\|_{L^{5/3}}\Bigl\|\frac{(\o^\e)^{2/5}}{\sqrt{|T(x)-T(y)|}}\Bigl\|_{L^{5/2}} \Bigl) \\
&\leq& C \sqrt{\sup_{B(0,R_2)\cap \Pi_\e}(|T_\e-T|)}.
\end{eqnarray*}

Therefore, the uniform convergence of $T_\e$ (Assumption \ref{3.1}) allows us to conclude.
\end{proof}

\begin{theorem}\label{5.6}
One has that $\F^\e u^\e \to u$ strongly in $L^2_{\loc}([0,\infty)\times \R^2)$, with
\begin{equation}
  \label{uomega}
u(x)=\dfrac{1}{2\pi} DT^t(x) \int_{\R^2} \Bigl(\dfrac{(T(x)-T(y))^\perp}{|T(x)-T(y)|^2}-\dfrac{(T(x)- T(y)^*)^\perp}{|T(x)- T(y)^*|^2}\Bigl)\o(y,t)dy + \a H(x).
\end{equation}
\end{theorem}

\begin{proof} We recall the explicit formula for $\F^\e u^\e$:
$$\F^\e u^\e=\dfrac{1}{2\pi} \F^\e(x) DT_\e^t(x) \int_{\Pi_\e} \Bigl(\dfrac{(T_\e(x)-T_\e(y))^\perp}{|T_\e(x)-T_\e(y)|^2}-\dfrac{(T_\e(x)- T_\e(y)^*)^\perp}{|T_\e(x)- T_\e(y)^*|^2}\Bigl)\o^\e(y,t)dy + \a \F^\e(x)H_\e(x).$$

Next we decompose:
\begin{eqnarray*}
(\F^\e u^\e-u)(x)&=&\frac{1}{2\pi} (DT_\e^t(x)-DT^t(x)) \F^\e(x) \int_{\Pi_\e} \Bigl(\dfrac{(T_\e(x)-T_\e(y))^\perp}{|T_\e(x)-T_\e(y)|^2}\\
&& {\hskip 7.5cm} -\dfrac{(T_\e(x)- T_\e(y)^*)^\perp}{|T_\e(x)- T_\e(y)^*|^2}\Bigl)\o^\e(y)dy \\
&+& \frac{1}{2\pi}DT^t(x)\F^\e(x)\int_{\Pi_\e} \Bigl(\dfrac{(T_\e(x)-T_\e(y))^\perp}{|T_\e(x)-T_\e(y)|^2} \\
&& {\hskip 5.5cm} -\dfrac{(T_\e(x)- T_\e(y)^*)^\perp}{|T_\e(x)- T_\e(y)^*|^2}\Bigl)\o^\e(y)(1-\F^\e(y)) dy\\
&+& \frac{1}{2\pi}DT^t(x)\F^\e(x)\int_{\R^2} \Bigl(\frac{(T_\e(x)-T_\e(y))^\perp}{|T_\e(x)-T_\e(y)|^2}-\frac{(T(x)-T(y))^\perp}{|T(x)-T(y)|^2}\Bigl)\F^\e(y)\o^\e(y)dy \\
&+& \frac{1}{2\pi}DT^t(x)(\F^\e(x)-1)\int_{\R^2}\dfrac{(T(x)-T(y))^\perp}{|T(x)-T(y)|^2}\F^\e(y)\o^\e(y)dy\\
&+& \frac{1}{2\pi}DT^t(x)\int_{\R^2}\dfrac{(T(x)-T(y))^\perp}{|T(x)-T(y)|^2}(\F^\e(y)\o^\e(y)-\o(y))dy \\
&-& \frac{1}{2\pi}DT^t(x)\F^\e(x)\int_{\R^2}\Bigl( \frac{(T_\e(x)-T_\e(y)^*)^\perp}{|T_\e(x)-T_\e(y)^*|^2}-\frac{(T(x)-T(y)^*)^\perp}{|T(x)-T(y)^*|^2}\Bigl)\F^\e(y)\o^\e(y) dy \\
&-& \frac{1}{2\pi}DT^t(x)(\F^\e(x)-1)\int_{\R^2}\dfrac{(T(x)-T(y)^*)^\perp}{|T(x)-T(y)^*|^2}\F^\e(y)\o^\e(y)dy \\
&-& \frac{1}{2\pi}DT^t(x)\int_{\R^2}\dfrac{(T(x)-T(y)^*)^\perp}{|T(x)-T(y)^*|^2}(\F^\e(y)\o^\e(y)-\o(y))dy \\
&+& \a \frac{1}{2\pi}\Bigl(\F^\e(x)DT_\e^t(x)\dfrac{T_\e(x)^\perp}{|T_\e(x)|^2}-DT^t(x)\dfrac{T(x)^\perp}{|T(x)|^2}\Bigl)\\
&\equiv& J_1+...+J_9.
\end{eqnarray*}

In every $J_i$, we use the fact that $DT$ is bounded in $L^3_{\loc}$ (see Proposition \ref{2.2}). We also use the estimates of the integrals $I_1$ and $I_2$ independently of $x$, $\e$, $t$ (see Lemma \ref{I_est}).

For $J_4$ and $J_7$, we remark that $(\F^\e(x)-1)\to 0$ in $L^6$, $DT$ is bounded in $L^3_{\loc}$ and the integral is bounded independently of  $x$, $\e$ and $t$ (see Remark \ref{remark4.3}), which is sufficient to conclude that $J_4$ and $J_7$ converge to zero in $L^2_{\loc}([0,\infty)\times \R^2)$.

A similar argument holds true for $J_1$, since $DT_\e\to DT$ in $L^3_{\loc}$ by Assumption \ref{3.1}.

For $J_2$, Lemma \ref{5.1} states that for fixed $x$, the fractions are bounded in $L^{4/3}(B(0,R_1))$ independently of $\e$. Moreover $\o^\e$ is bounded independently of  $t$, $\e$  and  $1-\F^\e(y)\to 0$ in $L^4$. Therefore, for fixed $x\notin \G$, the integral tends pointwise to $0$. Moreover this integral is bounded (see Lemma \ref{I_est}) and using the dominated convergence theorem, we can observe that it tends to $0$ in $L^6_{\loc}$. So, we have the convergence of $J_2$ to zero because $DT$ is bounded in $L^3_{\loc}$.

The convergence of $J_3$ to $0$ is a direct consequence of Lemma \ref{5.5}. Next, $DT$ belongs to $L^3_{\loc}(\R^2)$, and thanks to Lemma \ref{5.3}, we know that the integrals in $J_5$ and $J_8$ tend to zero in $L^2_{\loc}([0,\infty),L^6_{\loc}(\R^2))$. So $J_5$ and $J_8$ tend to zero in $L^2_{\loc}$.

We now go to $J_6$. Applying Lemma \ref{5.4} and reasoning as we did for the second term: for fixed $x$ and $t$, the integral tends pointwise to $0$  because $\o^\e$ is bounded in $L^4$ independently of $t$. Moreover, it is uniformly bounded by Lemma \ref{I_est}, and we can apply twice the dominated convergence theorem to obtain that the integral in $J_6$ converges to $0$ in  $L^2_{\loc}([0,\infty),L^6_{\loc}(\R^2))$. Using again the boundness of $DT$ in $L^3_{\loc}$ we get the desired conclusion for $J_6$.

The convergence of $J_9$ can be done more easily, because $1/|T|\leq 1$. Indeed we can decompose 
\begin{eqnarray*}
J_9 &=& \frac{\a}{2\pi}(DT_\e^t(x)-DT^t(x))\F^\e(x)\dfrac{T_\e(x)^\perp}{|T_\e(x)|^2}\\
&+& \frac{\a}{2\pi}\F^\e(x)DT^t(x)\Bigl(\dfrac{T_\e(x)^\perp}{|T_\e(x)|^2}-\dfrac{T(x)^\perp}{|T(x)|^2}\Bigl)\\
&+& \frac{\a}{2\pi}(\F^\e(x)-1)DT^t(x)\dfrac{T(x)^\perp}{|T(x)|^2},
\end{eqnarray*}
and the convergence to zero of $J_9$ is a direct consequence of points (i) and (iii) of the Assumption \ref{3.1}.
\end{proof}

The previous theorem provides an explicit formula expressing the limit velocity in terms of the limit vorticity. From this formula, we can deduce a few properties of the limit velocity $u$.

\begin{proposition}\label{remark}
Let $u$ be given as in Theorem \ref{5.6}. For fixed $t$, the velocity
\begin{itemize}
\item[i)] is continuous on $\R^2\setminus \G$.
\item[ii)] is continuous up to $\G\setminus \{-1;1\}$, with different values on each side of $\G$.
\item[iii)] blows up at the endpoints of the curve like $C/\sqrt{|x-1||x+1|}$, which belongs to $L^p_{\loc}$ for $p<4$.
\item[iv)] is tangent to the obstacle.
\end{itemize}
\end{proposition}

\begin{proof} To show that, we now prove that 
$$A(x)\equiv \int_{\R^2} \frac{T(x)-T(y)}{|T(x)-T(y)|^2}\o(y)dy \text{\ and\ } B(x)\equiv \int_{\R^2}\frac{T(x)-T(y)^*}{|T(x)-T(y)^*|^2}\o(y)dy$$ 
are continuous on $\R^2\setminus\G$ as $\o\in L^1\cap L^\infty$. As in the proof of Lemma \ref{I_est}, we change variables, we introduce $f(\y,t)=\o( T^{-1}(\y),t)J(\y)\h_{\{|\y|\geq 1\}}$ and $z=T(x)$. Then we have $\tilde A(z)\equiv A(x)=
\int_{|\y|\geq 1} \frac{z-\y}{|z-\y|^2}f(\y,t)d\y$, which is continuous for $f \in L^1\cap L^\infty$. In the same way we estimated $I_2$ in the proof of Lemma \ref{I_est}, we write
$$B(x)=\int_{|\y|\geq 2} \frac{z-\y^*}{|z-\y^*|^2}f(\y,t)d\y+\int_{1/2\leq |\th|\leq 1} \frac{z-\th}{|z-\th|^2}g(\th,t)d\th\equiv B_1(z)+B_2(z),$$
with $g(\th,t)=\o( T^{-1}(\th^*),t)J(\th^*)/|\th|^4$. As for $A$, we observe that $B_2$ is continuous. For $B_1$, taking a sequence $f_n\in C^\infty_c$ such that $f_n\to f(.,t)$ strongly in $L^1$, we see that $B_{1,n}\equiv \int_{|\y|\geq 2} \frac{z-\y^*}{|z-\y^*|^2}f_n(\y,t)d\y$ is continuous. As $|z-\y^*|\geq 1/2$, we can conclude, after remarking that $\|  B_n -  B_1\|_{L^\infty}\leq 2\|f_n-f\|_{L^1}$, that $B_1$ is continuous.

Moreover, $\tilde A$, $B_1$ and $B_2$ are continuous up to the boundary. As $A(x)=\tilde A(T(x))$ and $B(x)=B_1(T(x))+B_2(T(x))$, with $T$ continuous up to $\G\setminus \{-1;1\}$, with different values on each side of $\G$ (see Proposition \ref{2.2}), we proved i) and ii).

The blowing up at the endpoints is the consequence of the expression of $DT$ (see (\ref{T'})), and the fact that $A(x)$ and $B(x)$ is bounded by Lemma \ref{I_est}.

Finally, to show that the velocity is tangent to the obstacle, we do a simplified, but similar calculation to the one in Lemma \ref{4.6}. Indeed, as $|T(x)|=1$ on the curve $\G$, $\na |T(x)|$ is orthogonal to the curve. According to Proposition \ref{2.2}, $\na |T(x)|$ is continuous up to the curve $\G$ with different values on each side. Let $x\in \G\setminus\{-1;1\}$, then for a sequence $x_n\in \Pi$ which tends to $x$, we can make the same calculation than in the beginning of the proof of Lemma \ref{4.5}, to get:
\begin{eqnarray*}
u(x_n).\na |T(x_n)| &=& u(x_n)^\perp.\na^\perp |T(x_n)|\\
&=& -\dfrac{1}{2\pi} \int_{\R^2} \Bigl(\dfrac{T(x_n)-T(y)}{|T(x_n)-T(y)|^2}-\dfrac{T(x_n)- T(y)^*}{|T(x_n)- T(y)^*|^2}\Bigl)\o(y,t)dy \\
&& \times DT(x_n)DT^t(x_n)\frac{T(x_n)^\perp}{|T(x_n)|} + \a \frac{\na |T(x_n)|}{|T(x_n)|}.\na^\perp |T(x_n)|\\
&=& \dfrac{1}{2\pi|T(x_n)|}\det(DT)(x_n) \int_{\R^2} A_n(y)\o(y,t)dy 
\end{eqnarray*}
with 
\begin{eqnarray*}
A_n(y)&=&\dfrac{T(y).T(x_n)^\perp}{|T(x_n)-T(y)|^2}-\dfrac{ T(y)^*.T(x_n)^\perp}{|T(x_n)- T(y)^*|^2}\\
&=&\dfrac{|T(x_n)- T(y)^*|^2-|T(x_n)-T(y)|^2/|T(y)|^2}{|T(x_n)-T(y)|^2|T(x_n)- T(y)^*|^2}T(y).T(x_n)^\perp\\
&=& \dfrac{(1-1/|T(y)|^2)(|T(x_n)|^2-1)}{|T(x_n)-T(y)|^2|T(x_n)- T(y)^*|^2}T(y).T(x_n)^\perp.
\end{eqnarray*}
If $x_n$ tends to $x$ on one side of the curve, then $|T(x_n)|\to 1$ and $A_n(y)\to 0$. So $A_n \o(.,t)$ tends pointwise to zero, and as the integral is bounded by Remark \ref{remark4.3}, we can conclude by the dominated convergence theorem that $u(x).\na |T(x)|=0$ and that the velocity, on each side, is tangent to the curve.
\end{proof}

Therefore we have a weak $*$ limit for the vorticity, and a strong limit for the velocity. We now study the relation between $\curl u$ and $\o$.

\subsection{Calculation of curl and div of the velocity.}\

We first remark that $\diver u=0$, which is obvious since the velocity is the orthogonal gradient of a function. Indeed  $u=\nabla^\perp \p$ with $\p(x)=\int G_\pi(x,y)\o(y)dy+\frac{1}{2\pi} \log |T(x)|$. 

We now compute the curl of the limit velocity.

Recall that the curve $\G$ goes from $-1$ to $1$. Let $\overrightarrow{\t}=\G'/|\G'|$ the tangent vector of $\G$, $u_{up}$ the limit of $u(\G(s)+\r \overrightarrow{\t}^\perp)$ as $\r\to 0^+$ and $u_{down}$ the limit as $\r\to 0^-$.

\begin{lemma}\label{5.9} There exists a function $g_{\o}$ which depends on $\G$ and $\o$ such that 
$$\curl u=\o + g_{\o}(s) \d_\G,$$
in the sense of distributions.

Moreover $g_\o=(u_{down}-u_{up}).\overrightarrow{\t}$ which corresponds to the jump of the velocity on the curve.
\end{lemma}

\begin{proof} This proof is divided in two part. The first step consists to show that $\curl u-\o$ is concentrated on the curve $\G$, and we find in the second step the expression of $g_{\o}$.

We begin with $\curl H$. We remember that  $H=1/2\pi\ \nabla^\perp \log(T(x))$. Let $\f\in C^\infty_0(\R^2)$, we write:
\begin{eqnarray*}
\int_{\Pi} \curl(H)\f &=& -\int_{\Pi} H.\nabla^\perp \f \\
&=& -\frac{1}{2\pi}\int_{\Pi} \frac{T(x)}{|T(x)|^2}DT(x)\nabla \f(x)dx\\
&=& -\frac{1}{2\pi}\int_{D^c} \frac{z}{|z|^2}DT(T^{-1}(z))\nabla \f(T^{-1}(z))|\det\ T^{-1}|(z) dz
\end{eqnarray*}
where we changed variables $z=T(x)$. Since $T$ is holomorphic, we remark that 
$$\na(\f \circ T^{-1})(z)=DT^{-1^t}(z) \nabla\f(T^{-1}(z))=|\det\ T^{-1}|(z) DT(T^{-1}(z))\nabla \f(T^{-1}(z)).$$ We use the polar variables ($z=re^{i\th}$) to find
\begin{eqnarray*}
\int_{\Pi} \curl(H)\f &=& -\frac{1}{2\pi}\int_0^{2\pi}\int_1^{\infty} \na(\f \circ T^{-1})(z).\frac{z}{r}drd\th\\
&=& -\frac{1}{2\pi}\int_0^{2\pi}\int_1^{\infty} \frac{d}{dr}\Bigl(\f \circ T^{-1}(re^{i\th})\Bigl)drd\th\\
&=&  \frac{1}{2\pi}\int_0^{2\pi}\f \circ T^{-1}(\cos(\th),\sin(\th))d\th.
\end{eqnarray*}
The last integral can be written as an integral of $\f$ on the curve $\G$ with a certain weight.
 
We make the same calculation with the explicit formula of the velocity, and if we consider the translations $\t_1:z\mapsto z+T(y)$ and $\t_2:z\mapsto z+T(y)^*$, we obtain:

\begin{eqnarray*}
\int \curl(u).\f &=& -\frac{1}{2\pi}\int_{\R^2}\bigl(\int_{D^c-T(y)} \na(\f \circ T^{-1}\circ \t_1)(z).\frac{z}{|z|^2}dz\bigl) \o(y)dy   \\
&& +\frac{1}{2\pi}\int_{\R^2}\bigl(\int_{D^c-T(y)^*} \na(\f \circ T^{-1}\circ \t_2)(z).\frac{z}{|z|^2} dz\bigl) \o(y)dy  +\a \int \curl(H).\f \\
&=& \int_{\R^2}\f(y)\o(y)dy +\frac{1}{2\pi}\int_{\R^2}\int_{\th_0(y)}^{\th_1(y)}[\f\circ T^{-1}(A_{1,y}(\th))-\f\circ T^{-1}(A_{2,y}(\th))]d\th \o(y)dy \\
&& -\frac{1}{2\pi}\int_{\R^2}\int_0^{2\pi}\f\circ T^{-1}(A_{3,y}(\th)) \o(y)dy + \frac{\a}{2\pi}\int_0^{2\pi}\f \circ T^{-1}(\cos(\th),\sin(\th))d\th,
\end{eqnarray*}
 with $A_1$, $A_2$ and $A_3$ constructed like this: we consider the half-line starting at $T(y)$ having an angle $\th$ with the abscissa axis. There exist two angles $\th_0<\th_1$ such that the half-line is tangent to the unit circle. If we choose $\th\in(\th_0,\th_1)$, then the half-line intersects the circle in two points  $A_2$ and further $A_1$. For $A_3$ we do the same thing with the half-line starting at $T(y)^*$. In this case, we obtain each time an intersection with the unit circle.

Now, the difficulty is the change of variables. Indeed, as $T^{-1}(A_{i,y}(\th))\in \G$, we should change the variable $s=T^{-1} (A_{i,y}(\th))$ and we would obtain $\int_{\R^2}\int_{\G}\f(s)f_{i,y}(s)ds$, but this calculation is too complicated. In fact, we have just proved that $\curl u=\o+g_\o(s)\d_\G$, and we will directly find the expression of $g_\o$. For that, we consider the solution $v$ of the Green problem without obstacle. That is,
$$\diver v = 0 \text{\ and\ }\curl v =\o \text{\ in\ }\R^2.$$
The explicit formula is $v(x)=\int_{\R^2}\frac{(x-y)^\perp}{|x-y|^2}\o(y)dy$.
We denote by $w= u-v$. Then $\curl w=g_\o \d_\G$.

We now prove that $g_\o= (w_{down}-w_{up}).\overrightarrow{\t}$, with $w_{up}$ and $w_{down}$ defined in the same manner as $u_{up}$ and $u_{down}$. 

So, for $x\in\G\setminus\{-1;1\}$, there exists a small neighborhood $O$ of $x$, such $O\setminus \G$ is the union of two connected domains: $O_{up}$ and $O_{down}$. On the one side, $$\int_O \curl w\f=\int_{\G} \f(s)g_\o(s)ds.$$
On the other side, 
\begin{eqnarray*}
\int_O \curl w\ \f &=& -\int_O w. \na^\perp\f=-\int_{O_{up}} w.\na^\perp\f-\int_{O_{down}} w.\na^\perp\f \\
&=& \int_{O_{up}} \curl w\ \f - \int_{\pd O_{up}}\f w. \overrightarrow{\t} +\int_{O_{down}} \curl w\ \f - \int_{\pd O_{down}} \f w  .\overrightarrow{\t} \\
&=& - \int_{\G} \f w_{up} . \overrightarrow{\t} +  \int_{\G} \f w_{down} . \overrightarrow{\t}
\end{eqnarray*}
because $w=u-v$ is continuous on $\R^2\setminus \G$.

As we want, we have $g_\o= (w_{down}-w_{up}).\overrightarrow{\t}$. Moreover, adding the regular part $v$, we have 
\begin{equation}
\label{g_o}
g_\o=\curl w= (u_{down}-u_{up}).\overrightarrow{\t}.
\end{equation}

Therefore, we obtain:
$$\int \f\curl(u) = \int_{\R^2} \f(y)\o(y)dy + \int_{\G}\f(s) g_\o(s). { ds} ,$$
with $g_\o$, bounded outside the endpoints, and equivalent at the endpoints to $$\frac{1}{\pi}\frac{A(\pm 1)}{\sqrt{|s-1||s+1|}}$$ with 
$$A(\pm 1)=C_{\pm 1} \Bigl|\int_{\R^2} \Bigl(\frac{(T(\pm 1)-T(y))^\perp}{|T(\pm 1)-T(y)|^2}-\frac{(T(\pm 1)- T(y)^*)^\perp}{|T(\pm 1)- T(y)^*|^2}\Bigl)\o(y,t)dy + \a T(\pm 1)^\perp \Bigl| $$
which is bounded. Indeed, we can prove that $g_\o$ is continuous as we prove that $u$ is continuous in Proposition \ref{remark}.
\end{proof} 

Getting a simplification is really hard, and we remark that we can not obtain a result like $\curl u=\o+g(s)\d_\G$. Even in the simpler case of the segment, the calculation of $g_\o$ does not give a good result. However, we can explicit the calculation of $\curl(H)$ in the case where the curve $\G$ is the segment $[-1,1]$: 
 \newline$T^{-1}(\cos(\th),\sin(\th))= (\cos(\th),0)$ and we change the variable $\y=\cos\ \th$ to have $$\int \curl(H)\f = 2.\frac{1}{2\pi}\int_0^{\pi}\f(cos(\th),0)d\th = \frac{1}{\pi}\int_{-1}^1\frac{\f(\y,0)}{\sqrt{1-\y^2}}d\y.$$

Moreover, we remark that $g_{\o}$, like the velocity, blows up at the endpoints of the curve $\G$ as the inverse of the square root of the distance.
 
\subsection{The asymptotic vorticity equation in $\R^2$} \ 

We begin by observing that the sequence $\{\F^\e\o^\e\}$ is bounded in $L^\infty([0,\T],L^4)$, then, passing to a subsequence if necessary, we have $$\F^\e\o^\e\rightharpoonup \o\text{, weak-$*$ in }L^\infty([0,\T],L^4).$$ 
We already have a limit velocity.

The purpose of this section is to prove that $u$ and $\o$ verify, in an appropriate sense, the system:
\begin{equation}
\left\lbrace \begin{aligned}
\label{tour_equa}
&\o_t+u.\na\o=0, & \text{ in }\R^2\times(0,\infty) \\
& \diver u=0 \text{ and }\curl u=\o+g_\o(s)\d_\G, &\text{ in }\R^2\times[0,\infty) \\
& |u|\to 0, &\text{ as }|x|\to \infty \\
& \o(x,0)=\o_0(x), &\text{ in }\R^2.
\end{aligned} \right .
\end{equation}
where $\d_\G$ is the Dirac function along the curve and $g_\o$ is given in (\ref{g_o}).

\begin{definition} The pair $(u,\o)$ is a weak solution of the previous system if
\begin{itemize}
\item[(a)] for any test function $\f\in C^\infty_c([0,\infty)\times\R^2)$ we have 
$$\int_0^\infty\int_{\R^2}\f_t\o dxdt +\int_0^\infty \int_{\R^2}\na\f.u\o dxdt+\int_{\R^2}\f(x,0)\o_0(x)dx=0,$$
\item[(b)] we have $\diver u=0$ and $\curl u=\o+g_\o\d_\G$ in the sense of distributions of $\R^2$, with $|u|\to 0$ at infinity.
\end{itemize}
\end{definition}

\begin{theorem}\label{5.11} The pair $(u,\o)$ obtained at the beginning of this subsection is a weak solution of the previous system.
\end{theorem}
\begin{proof}
The second point of the definition is directly verified by the previous section and by the estimate of Subsection \ref{biot} about the far-field behavior. Indeed, the velocity $u$ verifies $|u|\to 0$ at infinity, thanks to the explicit expressions for $K[\o]$ and $H$, using the uniform compact support of $\o$.

Next, we introduce an operator $I_\e$, which for a function $\f\in C^\infty_0([0,\infty)\times \R^2)$ gives:
$$I_\e[\f]\equiv \int_0^\infty\int_{\R^2} \f_t(\F^\e)^2\o^\e dxdt+\int_0^\infty\int_{\R^2} \na\f.(\F^\e u^\e)(\F^\e\o^\e) dxdt.$$
To prove that $(u,\o)$ is a weak solution, we will show that
\begin{itemize} 
\item[(i)] $I_\e[\f]+\int_{\R^2}\f(x,0)\o_0(x)dx\to 0$ as $\e\to 0$
\item[(ii)] $I_\e[\f]\to \int_0^\infty\int_{\R^2}\f_t\o dxdt +\int_0^\infty\int_{\R^2}\na\f.u\o dxdt$ as $\e\to 0$.
\end{itemize}
Clearly these two steps complete the proof.

We begin by showing (i). As $u^\e$ and $\o^\e$ verify  $(\ref{E_e})$, it can be easily seen that
$$\int_0^\infty\int_{\R^2} \f_t(\F^\e)^2\o^\e dxdt =-\int_0^\infty\int_{\R^2} \na(\f (\F^\e)^2).u^\e\o^\e dxdt-\int_{\R^2} \f(x,0)(\F^\e)^2(x)\o_0(x) dx$$
Thus we compute
$$ I_\e[\f] = -2\int_0^\infty\int_{\R^2} \f\na\F^\e. u^\e(\F^\e\o^\e) dxdt-\int_{\R^2} \f(x,0)(\F^\e)^2(x)\o_0(x) dx $$

We have:
$$\Bigl|I_\e[\f]+\int_{\R^2} \f(x,0)(\F^\e)^2(x)\o_0(x) dx \Bigl|\leq 2 \|\F^\e\o^\e\|_{L^\infty (L^\infty)}\|\f\|_{L^1(L^\infty)}\| u^\e.\na\F^\e\|_{L^\infty(L^1)}\to 0,$$
as $\e\to 0$ by Lemma \ref{4.5}. This shows (i) for all $\e$ sufficiently small such that $(\F^\e)^2(x)\o_0=\o_0$ since the support of $\o_0$ does not intersect the curve.

For (ii), the linear term presents no difficulty. The second term consists of the weak-strong pair vorticity-velocity:
\begin{eqnarray*}
\Bigl|\int\int\na\f.(\F^\e u^\e)(\F^\e\o^\e)-\int\int\na\f. u\o\Bigl| &\leq& \Bigl|\int\int\na\f.(\F^\e u^\e-u)(\F^\e\o^\e)\Bigl| \\
&&+\Bigl|\int\int\na\f.u(\F^\e\o^\e-\o)\Bigl|.
\end{eqnarray*}
$\F^\e u^\e\to u$ strongly in $L^2([0,\T],L^2_{\loc}(\R^2))$ thanks to Theorem \ref{5.6}. So the first term tends to zero because $\F^\e\o^\e$ is bounded in $L^\infty([0,\infty),L^2(\R^2))$. In the same way, the second term tends to zero because $\F^\e\o^\e\rightharpoonup \o\text{ weak-$*$ in }L^\infty([0,\T],L^4(\R^2))$ and $u\in L^\infty([0,\T],L^2_{\loc}(\R^2))$.
\end{proof}

\subsection{The asymptotic velocity equation in $\R^2$} \ 

As the function $u$ is bounded in $L^2$, we can write the vorticity equation more simply than in \cite{ift_lop}. The main calculation of this subsection can be found in \cite{ift_lop}. 

We begin by introducing $v(x)=\int K(x-y)\o(y)dy$ with $K(x)=\frac{1}{2\pi}\frac{x^\perp}{|x|^2}$, the solution without obstacle of
\begin{equation*}
\left\lbrace\begin{aligned}
\diver v &=0 &\text{ on } \R^2, \\
\curl v &=\o &\text{ on } \R^2, \\
\lim_{|x|\to\infty}|v|&=0.
\end{aligned}\right.
\end{equation*}

This velocity is bounded, and we denote the perturbation by $w=u-v$, which is bounded in $L^p_{\loc}$ for $p<4$, and it verifies 
\begin{equation*}
\left\lbrace\begin{aligned}
\diver w &=0 &\text{ on } \R^2, \\
\curl w &=g_\o(s)\d_\G &\text{ on } \R^2, \\
\lim_{|x|\to\infty}|w|&=0.
\end{aligned}\right.
\end{equation*}

We now prove that $v$ verifies the following equation:

\begin{equation}
\label{vit_equa}
\begin{cases}
v_t+v.\na v+ v.\na w + w.\na v - v(s)^\perp\tilde g_v(s)\d_\G=-\na p, & \text{ in }\R^2\times(0,\infty)  \\
\diver v=0, & \text{ in }\R^2\times(0,\infty) \\
w(x)=\frac{1}{2\pi} \int_\G \frac{(x-s)^\perp}{|x-s|^2}\tilde g_v(s). {\bf ds}, & \text{ in }\R^2\times(0,\infty)  \\
v(x,0)=K[\o_0], & \text{ in }\R^2.
\end{cases}
\end{equation}
with $\tilde g_v=g_{\curl v}$.

In order to prove the equivalence of (\ref{tour_equa}) and (\ref{vit_equa}) it is sufficient to show that
\begin{equation}
\label{equiv}
\curl[v.\na w + w.\na v - v(s)^\perp\tilde g_v(s)\d_\G]=\diver(\o w)
\end{equation}
for all divergence free fields $v\in W^{1,p}_{\loc}$, with some $p>2$. Indeed, if (\ref{equiv}) holds, then we get for $\o=\curl v$ 
\begin{eqnarray*}
0 &=& -\curl \na p=\curl[v_t+ v.\na v + v.\na w + w.\na v - v(s)^\perp\tilde g_v(s)\d_\G] \\
&=& \o_t + v.\na \o + w.\na \o= \o_t + u.\na \o=0
\end{eqnarray*}
so relation (\ref{tour_equa}) holds true. And vice versa, if (\ref{tour_equa}) holds then we deduce that the left hand side of (\ref{vit_equa}) has zero curl so it must be a gradient.

We now prove (\ref{equiv}). As $ W^{1,p}_{\loc} \subset \mathcal{C}^0$, $v(s)$ is well defined. Next, it suffices to prove the equality for smooth $v$, since we can pass to the limit on a subsequence of smooth approximations of $v$ which converges strongly in $W^{1,p}_{\loc}$ and $\mathcal{C}^0$. Now, it is trivial to check that, for a $2\times 2$ matrix $A$ with distribution coefficients, we have
$$\curl \diver A =\diver \begin{pmatrix} \curl C_1 \\ \curl C_2 \end{pmatrix}$$
where $C_i$ denotes the $i$-th column of $A$.
For smooth $v$, we deduce
\begin{eqnarray*}
\curl[v.\na w + w.\na v] &=& \curl \diver(v\otimes w+w \otimes v)\\
&=& \diver\begin{pmatrix} \curl(vw_1)+\curl(wv_1) \\ \curl(vw_2)+\curl(wv_2) \end{pmatrix} \\
&=& \diver(w\ \curl v+v.\na^\perp w+ v\ \curl w+w.\na^\perp v).
\end{eqnarray*}

It is a simple computation to check that
$$ \diver(v.\na^\perp w+w.\na^\perp v) = v.\na^\perp \diver w+ w.\na^\perp \diver v + \curl v\ \diver w+ \curl w\ \diver v.$$

Taking into account that we have free divergence fields, we can finish by writing
\begin{eqnarray*}
\curl[v.\na w + w.\na v] &=& \diver(w\ \curl v +v \tilde g_v(s)\d_\G) \\
&=& \diver(w\ \curl v) + \curl[v(s)^\perp \tilde g_v(s)\d_\G].
\end{eqnarray*}
which proves (\ref{equiv}).

Now, we write a formulation for the velocity $u$, by replacing $v$ by $u-w$ to obtain in $\R^2$
$$u_t + u.\na u = -\na p+ w_t+w.\na w + v(s)^\perp \tilde g_v(s) \d_\G.$$

However, since $\curl w=0$, we can remark that $\curl[w_t+w.\na w+ v(s)^\perp \tilde g_v(s) \d_\G]=0$ in $\R^2\setminus \G$.

\subsection{Formulation on $\R^2\setminus \G$} \ 

We can obtain directly an equation for $u$ on $\R^2\setminus \G$ by
passing to the limit $\e\to 0$. We multiply the velocity equation
(\ref{E_e}) by some divergence-free test vector field $\f\in
C_c^\infty(\R^+\times(\R^2\setminus \G))$ and assume that $\e$ is small
enough such that the support of $\f$ is contained in $\Pi_\e$ and do
not intersected the support of $\na\F^\e$. After integration,
$$ \int_0^\infty\int u^\e.\pd_t\f + \int_{\R^2\setminus \G} u^\e(0,.).\f(0,.) + \int_0^\infty\int (u^\e\otimes u^\e).\na \f =0,$$
which easily pass to the limit since $u^\e\to u$ strongly in
$L^2_{\loc}$ by Theorem \ref{5.6}. Indeed, we can prove easily that $u^\e_0\to u_0$ in $L^2_{\loc}(\R^2)$ thanks to the proof of Theorem \ref{5.6}. Therefore, the above relation holds
true with $u$ instead of $u^\e$ and this is the formulation in the
sense of distributions of the Euler equation in $\R^2\setminus \G$:

\begin{equation}
\left\lbrace \begin{aligned}
&u_t+u.\na u=-\na p, & \text{\ in\ } \R^2\setminus \G\times(0,\infty) \\
&\diver u=0 & \text{\ in\ } \R^2\setminus \G\times[0,\infty) \\
&u.\hat{n}=0 &\text{\ on\ } \G\times[0,\infty) \\
& |u|\to 0, &\text{ as }|x|\to \infty \\
&u(x,0)=F(\o_0) & \text{\ in\ } \R^2\setminus \G 
\end{aligned} \right .
\end{equation}
where $F$ is the formula from Theorem \ref{5.6} expressing explicitly
the velocity in terms of vorticity and circulation.

\section*{Acknowledgments}

I want to thank some researchers for the help offered with the subject of complex analysis, which allows the generalization to any curve $\G$. In particular Etienne Ghys and  Alexei Glutsyuk who searched with me a proof using quasi-conformal mappings for the straightening up of the curve by a holomorphism, $C^\infty$ up to the boundary. I would like to give many thanks to Xavier Buff for suggesting the use of $\tilde T_\pm$. 

I also want to thank M.C. Lopes Filho and H.J. Nussenzveig Lopes for several interesting and helpful discussions.

\section*{List of notations}

\subsection*{Domains:}\

$D\equiv B(0,1)$ the unit disk.

$S\equiv \pd D$.

$\G$ is a Jordan arc (see Proposition \ref{2.2}).

$\Pi \equiv \R^2\setminus \G$.

$\O_0$ is a bounded, open, connected, simply connected subset of the plane.

$\G_0\equiv \pd\O_0$ is a $C^\infty$ Jordan curve.

$\Pi_0\equiv \overline{\R^2\setminus \O_0}$

$\O_\e$ is a family of a domain, verifying the same properties of $\O_0$, such as $\O_\e\to \G$ as $\e\to 0$.

$\G_\e\equiv \pd\O_\e$ and $\Pi_\e\equiv \overline{\R^2\setminus \O_\e}$.

\subsection*{Functions:}\

$\o_0$ is the initial vorticity ($C^\infty_c(\Pi)$).

$\g$ is the circulation of $u_0^\e$ on $\G_\e$ (see Introduction).

$(u^\e,\o^\e)$ is the solution of the Euler equations on $\Pi_\e$.

$T$ is a biholomorphism between $\Pi$ and $\inte\ D^c$ (see Proposition \ref{2.2}).

$T_0$ is a biholomorphism between $\Pi_0$ and $\inte\ D^c$.

$T_\e$ is a biholomorphism between $\Pi_\e$ and $\inte\ D^c$ (see Assumption \ref{3.1}).

$K^\e$ and $H^\e$ are given in \eqref{K} and \eqref{H}

$K^\e[\o^\e](x)\equiv \int_{\Pi_\e} K^\e(x,y) \o^\e(y)dy$.

$\F^\e$ is a cutoff function in a small $\e$-neighborhood of $\O_\e$ (see Subsection \ref{cutoff}).


\begin{thebibliography}{99}

\bibitem{ift_lop} Iftimie D., Lopes Filho M.C. and Nussenzveig Lopes H.J., {\it Two Dimensional Incompressible Ideal Flow Around a Small Obstacle}, Comm. Partial Diff. Eqns. 28 (2003), no. 1$\&$2, 349-379.

\bibitem{beel_krantz} Beel S. and Krantz S.G., {\it Smoothness to the boundary of conformal maps}, Rocky Mt J Math 1987; 17(1):23-40.

\bibitem{pomm_1} Pommerenke C., {\it Univalent functions}, Vandenhoeck $\&$ Ruprecht, 1975.

\bibitem{pomm_2} Pommerenke C., {\it Boundary behaviour of conformal maps}, BerlinNew York: Springer-Verlag, 1992.

\bibitem{ift} Iftimie D., {\it Evolution de tourbillon \`a support compact}, Actes du Colloque de Saint-Jean-de-Monts, 1999.

\bibitem{kiku} Kikuchi K., {\it Exterior problem for the two-dimensional Euler equation}, J Fac Sci Univ Tokyo Sect 1A Math 1983; 30(1):63-92.

\bibitem{lop} Lopes Filho M.C., {\it Vortex dynamics in a two dimensional domain with holes and the small obstacle limit}, to appear, SIAM Journal on Mathematical Analysis, 2006.

\bibitem{ift_lop_2} Iftimie D., Lopes Filho M.C. and Nussenzveig Lopes H.J., {\it Two Dimensional Incompressible Viscous Flow Around a Small Obstacle}, Math. Ann., {\bf 336} (2006), 449-489.

\end{thebibliography}
\end{document}